\documentclass[11pt]{article}
\usepackage[centertags]{amsmath}
\usepackage{amsfonts}
\usepackage{amssymb}
\usepackage{amsthm,color}
\usepackage{newlfont}

\newtheorem{Theorem}{Theorem}[section]
\newtheorem{Definition}[Theorem]{Definition}

\newtheorem{Lemma}[Theorem]{Lemma}
\newtheorem{Corollary}[Theorem]{Corollary}
\newtheorem{Remark}[Theorem]{Remark}

\newtheorem{Hypothesis}{Hypothesis}

\numberwithin{equation}{section}

\begin{document}

\def\le{\left}
\def\r{\right}
\def\cost{\mbox{const}}
\def\a{\alpha}
\def\d{\delta}
\def\ph{\varphi}
\def\e{\epsilon}
\def\la{\lambda}
\def\si{\sigma}
\def\La{\Lambda}
\def\B{{\cal B}}
\def\A{{\mathcal A}}
\def\L{{\mathcal L}}
\def\O{{\mathcal O}}
\def\bO{\overline{{\mathcal O}}}
\def\F{{\mathcal F}}
\def\K{{\mathcal K}}
\def\H{{\mathcal H}}
\def\D{{\mathcal D}}
\def\C{{\mathcal C}}
\def\M{{\mathcal M}}
\def\N{{\mathcal N}}
\def\G{{\mathcal G}}
\def\T{{\mathcal T}}
\def\R{{\mathbb R}}
\def\I{{\mathcal I}}

\def\bw{\overline{W}}
\def\phin{\|\varphi\|_{0}}
\def\s0t{\sup_{t \in [0,T]}}
\def\lt{\lim_{t\rightarrow 0}}
\def\iot{\int_{0}^{t}}
\def\ioi{\int_0^{+\infty}}
\def\ds{\displaystyle}
\def\pag{\vfill\eject}
\def\fine{\par\vfill\supereject\end}
\def\acapo{\hfill\break}

\def\beq{\begin{equation}}
\def\eeq{\end{equation}}
\def\barr{\begin{array}}
\def\earr{\end{array}}
\def\vs{\vspace{.1mm}   \\}
\def\rd{\reals\,^{d}}
\def\rn{\reals\,^{n}}
\def\rr{\reals\,^{r}}
\def\bD{\overline{{\mathcal D}}}
\newcommand{\dimo}{\hfill \break {\bf Proof - }}
\newcommand{\nat}{\mathbb N}
\newcommand{\E}{\mathbb E}
\newcommand{\Pro}{\mathbb P}
\newcommand{\com}{{\scriptstyle \circ}}
\newcommand{\reals}{\mathbb R}

\def\Amu{{A_\mu}}
\def\Qmu{{Q_\mu}}
\def\Smu{{S_\mu}}
\def\H{{\mathcal{H}}}
\def\Im{{\textnormal{Im }}}
\def\Tr{{\textnormal{Tr}}}
\def\E{{\mathbb{E}}}
\def\P{{\mathbb{P}}}

\title{Smoluchowski-Kramers approximation and  large deviations for infinite dimensional non-gradient systems with  applications to the exit problem}
\author{Sandra Cerrai, Michael Salins\\
\vspace{.1cm}\\
Department of Mathematics\\
 University of Maryland\\
College Park\\
 Maryland, USA
}

\date{}

\maketitle

\begin{abstract}
  In this paper, we study the  quasi-potential for a general class of  damped semilinear stochastic wave equations.  We show that, as the density of the mass converges to zero,   the infimum of the quasi-potential with respect to all possible velocities converges to the quasi-potential of the corresponding stochastic heat equation, that one obtains from the zero mass limit.
This shows in particular that the Smoluchowski-Kramers approximation is not only valid for small time, but, in the zero noise limit regime, can be used to approximate long-time behaviors such as exit time and exit place from a basin of attraction.
\end{abstract}

\section{Introduction}

In the present paper, we are dealing with  the following stochastic wave equation in a bounded regular domain $D\subset \reals^d$, with $d\geq 1$,
\begin{equation}
\label{semilinear-wave-eq-intro}
\le\{\begin{array}{l}
\ds{     \mu \frac{\partial^2 u^\mu_\e}{{\partial t}^2}(t,\xi) = \Delta u^\mu_\e(t,\xi) - \frac{\partial u^\mu_\e}{\partial t}(t,\xi)
  + B(u^\mu_\e(t))(\xi) + \sqrt{\epsilon}\,\frac{\partial w^Q}{\partial t}(t,\xi),\ \ \ \xi \in D,}\\
  \vs
  \ds{
     u^\mu_\e(0,\xi) = u_0(\xi),\ \ \ \
     \frac{\partial u^\mu_\e}{\partial t}(0,\xi) = v_0(\xi),\ \ \xi \in\,D,\ \ \ \ \      u^\mu_\e(t,\xi) = 0,\ \  \xi \in \partial D.}
     \end{array}\r.
     \end{equation}
Here $\partial w^Q/\partial t$ is a cylindrical Wiener process, white in time and colored in space, with covariance $Q^2$, and $\mu$ and $\e$ are small positive constants.

As a consequence of the Newton law, we may interpret the solution $u_{\e}^{\mu}(t,\xi)$ of equation \eqref{semilinear-wave-eq-intro} as the displacement field of the particles of a material continuum in the domain $D$, subject to a random external force field $\sqrt{\epsilon}\partial w^Q/\partial t(t,\xi)$ and a damping force proportional to the velocity field $\partial u_{\e}^{\mu}/\partial t(t,\xi)$. The Laplacian describes interaction forces between neighboring particles, in presence of a non-linear reaction described by $B$. The constant $\mu$ represents the constant density of the particles.

 In \cite{smolu2} and \cite{smolu1}, it has been proven that, for fixed $\epsilon>0$, as the density $\mu$ converges to $0$, the solution $u_{\e}^{\mu}(t)$ of problem \eqref{semilinear-wave-eq-intro} converges to the solution $u_\e(t)$ of the stochastic first order equation
 \begin{equation} \label{semilinear-heat-eq-intro}
\le\{\begin{array}{l}
\ds{     \frac{\partial u_\e}{\partial t}(t,\xi) = \Delta u_\e(t,\xi) + B(Òu_\e(t))(\xi) + \sqrt{\epsilon}\, \frac{\partial w^Q}{\partial t}(t,\xi),\ \ \ \ \xi \in D,} \\
 \vs
 \ds{    u_\e(0,\xi) = u_0(\xi),\ \ \ \ \xi \in\,D,\ \ \
     u_\e(t,\xi) = 0,\ \ \ \xi \in \partial D,}
\end{array}  \r.
 \end{equation}
uniformly for $t$ on fixed  intervals.
More precisely, we have shown that for any $\eta>0$ and $T>0$
\begin{equation}
\label{introsk}
\lim_{\mu\to 0}\,\Pro\le(\sup_{t \in\,[0,T]}|u^\mu_\e(t)-u_\e(t)|_H>\eta\r)=0.
\end{equation}
Such an approximation is known as the Smoluchowski-Kramers approximation.

\medskip

Once one has proved the validity of \eqref{introsk}, an  important question arises: how do some relevant asymptotic properties of the second and the  first order systems  compare, with respect to the small mass asymptotic?  In  \cite{f} and \cite{cf} the case of systems with a finite number of degrees of freedom has been studied and    the large deviation estimates, with the exit problem
from a domain,  various averaging procedures, the Wong-Zakai approximation, and  homogenization
have been compared. It has been proven  that in some cases the
two asymptotics do match together properly and in other cases they
don't.

In \cite{smolu2}, where the validity of the Smoluchowskii-Kramers approximation for SPDEs has been approached for the first time,  the long time behavior
of equations \eqref{semilinear-wave-eq-intro} and \eqref{semilinear-heat-eq-intro} has been compared, under the assumption that the two systems are of gradient type.
Actually, in the case of white noise in space and time (that is $Q=I$ and hence $d=1$) an
explicit expression for the Boltzman distribution of the process
$z^\mu_\e(t):=(u_\e^{\mu}(t), \partial u_\e^{\mu}/\partial t(t))$ in the phase space $\H:=L^2(0,1)\times H^{-1}(0,1)$ has been
given. Of course, since in the functional space $\H$ there is no
analogous of the Lebesgue measure,
 an auxiliary Gaussian measure has been introduced, with respect to which
 the density of the Boltzman distribution has  been written down. This
auxiliary Gaussian measure is the stationary measure of the linear
wave equation related to problem \eqref{semilinear-wave-eq-intro}.
In particular, it has been shown  that  the first marginal of the invariant measure
associated with the process $z_\e^{\mu}(t)$ does not depend on $\mu$ and  coincides with the
invariant measure of the process $u_\e(t)$, defined as the unique
solution of the heat equation \eqref{semilinear-heat-eq-intro}.

\medskip

In the present paper, we are interested in comparing the small noise asymptotics, as $\e\downarrow 0$, for system \eqref{semilinear-wave-eq-intro} and system \eqref{semilinear-heat-eq-intro}. Actually, we want to show that the Smoluchowski-Kramers approximation, that works on finite time intervals, is good also in the large deviations regime. More precisely, we want to compare the quasi-potential $V^\mu(x,y)$ associated with \eqref{semilinear-wave-eq-intro}, with the quasi-potential $V(x)$ associated with \eqref{semilinear-heat-eq-intro}, and we want to show that for any closed set $N \subset L^2(D)$ it holds
\begin{equation}
\label{intro10}
\lim_{\mu\to 0}\,\inf_{x \in N}\,V_\mu(x):=\lim_{\mu\to 0}\,\inf_{x \in N}\,\inf_{y \in\,H^{-1}(D)}V^\mu(x,y)=\inf_{x \in N}\,V(x).
\end{equation}

This means that taking first the limit as  $\e\downarrow 0$ (large deviation) and then taking the limit as $\mu\downarrow 0$ (Smoluchowski-Kramers approximation) is the same as first taking the limit as  $\mu\downarrow 0$ and then   as  $\e\downarrow 0$. In particular, this result provides a rigorous mathematical justification of what is done in  applications, when, in order to study rare events and transitions between metastable states for the more complicated system \eqref{semilinear-wave-eq-intro}, as well as exit times from basins of attraction and the corresponding exit places,  the relevant quantities associated with the large deviations for system \eqref{semilinear-heat-eq-intro} are considered.

\medskip

In our previous paper \cite{sal}, we have addressed this problem in  the particular case system \eqref{semilinear-wave-eq-intro} is of gradient type, that is
\begin{equation}
\label{intro19}
B(x)=-Q^2 D\!F(x),\ \ \ x \in\,L^2(D),\end{equation}
for some $F:L^2(D)\to \reals$, where $Q^2$ is the covariance of the Gaussian random perturbation. This applies for example to the linear case (that is $B=0$) in any space dimension or to the case
\[B(x)(\xi)=b(\xi,x(\xi)),\ \ \ \ \xi \in\,D,\]
when $D=[0,L]$ and $Q=I$. In \cite{sal} we have shown that,  if \eqref{intro19} holds, then    for any $\mu>0$
 \begin{equation}
 \label{intro11}
V^\mu(x,y) = \left|(-\Delta)^{1/2}Q^{-1} x \right|_{L^2(D)}^2 + 2F(x) + \mu \left|Q^{-1} y \right|_{L^2(D)}^2,
  \end{equation}
  for any $(x,y) \in\,D((-\Delta)^{1/2}Q^{-1})\times D(Q^{-1}).$
Therefore, as
\[V(x) = \left|(-\Delta)^{\frac{1}{2}} Q^{-1} x \right|_H^2 + 2F(x),\ \ \ \ x \in\,D((-\Delta)^{1/2}Q^{-1}), \]
from \eqref{intro11} we have concluded  that for any $\mu>0$,
\begin{equation}
\label{intro20}   V_\mu(x):= \inf_{y \in H^{-1}(D)} V^\mu(x,y) = V^\mu(x,0) = V(x),\ \ \ \ x \in\,D((-\Delta)^{1/2}Q^{-1}).
\end{equation}
In particular, this means that $V_\mu(x)$ does not just coincide with $V(x)$ at the limit, as in \eqref{intro10}, but for any fixed $\mu>0$.

\medskip

In the general non-gradient case that we are considering in the present paper, the situation is considerably more delicate and we cannot expect anything explicit as in \eqref{intro11}.  The lack of an explicit expression for $V^\mu(x,y)$ and $V(x)$ makes the proof of  \eqref{intro10} much more difficult and requires the introduction of new arguments and techniques.

The first key idea in order to prove \eqref{intro10} is to characterize $V^\mu(x,y)$ as the minimum value for a suitable functional. We recall that the quasi-potential $V^\mu(x,y)$ is defined as the minimum energy required to the system to go from the asymptotically stable equilibrium $0$ to the point $(x,y) \in\,\H$, in any time interval. Namely
\[V^\mu(x,y) = \inf \left\{ I^\mu_{0,T}(z) \ ;\  z(0) = 0,\  z(T) =(x,y),\  T > 0 \right\},\]
where
\[I^{\mu}_{0,T}(z) =\frac{1}{2}  \inf \left\{ | \psi|_{L^2((0,T);H)}^2\,:\, z=z^\mu_{\psi}\right\},
\] is the large deviation action functional
 and $z^\mu_{\psi} = (u^\mu_\psi, \partial u^\mu_\psi/\partial t)$ is a mild solution of the skeleton equation associated with equation  \eqref{semilinear-wave-eq-intro}, with control $\psi \in\,L^2((0,T);H)$,
\begin{equation}
\label{intro12}
\mu \frac{\partial^2 u^\mu_\psi}{\partial t^2}(t) = \Delta u^\mu_\psi(t) - \frac{\partial u^\mu_\psi}{\partial t}(t) + B(u^\mu_\psi(t)) + Q \psi(t),\ \ \ \ t \in\,[0,T].
\end{equation}
By working thoroughly with the skeleton equation \eqref{intro12}, we  show  that,
for small enough $\mu>0$,
\begin{equation}
\label{intro13}
  V^\mu(x,y) = \min \left\{ I^\mu_{-\infty,0}(z): \lim_{t \to -\infty} |z(t)|_{\H} =0,\  z(0)=(x,y) \right\}.\end{equation}
In particular, we get that the level sets of $V^\mu$ and $V_\mu$ are compact in $\H$ and $L^2(D)$, respectively. Moreover, we show that both $V^\mu$ and $V_\mu$ are well defined and continuous in suitable Sobolev spaces of functions. 
We would like to stress that in \cite{cerrok} a result analogous to \eqref{intro13} has been proved for equation \eqref{semilinear-heat-eq-intro} and $V(x)$, in terms of the corresponding functional $I_{-\infty,0}$. In both cases, the proof is highly non trivial, due to the degeneracy of the associated control problems, and requires a detailed analysis of the optimal regularity of the solution of the skeleton equation \eqref{intro12}. 

The second key idea is based on the fact  that,  as in \cite{cf} where the finite dimensional case is studied, for all functions $z \in\,C((-\infty,0];\H)$ that are regular enough,
\begin{equation} \label{intro14}
\begin{array}{l}
\ds{  I^\mu_{-\infty} (z) = I_{-\infty}(\varphi) + \frac{\mu^2}{2} \int_{-\infty}^0 \left| Q^{-1} \frac{\partial^2\varphi}{\partial t^2}(t) \right|_H^2 dt}\\
\vs
\ds{  + \mu \int_{-\infty}^0 \left< Q^{-1} \frac{\partial^2\varphi}{\partial t^2}(t), Q^{-1} \left( \frac{\partial\varphi}{\partial t}(t) - A \varphi(t) - B(\varphi(t)) \right) \right>_H dt=:I_{-\infty}(\varphi)+J^\mu_{-\infty}(z),}
\end{array}
\end{equation}
where $\varphi(t)=\Pi_1 z(t)$. Thus, if $\bar{z}^\mu$ is the minimizer of $V_\mu(x)$, whose existence is guaranteed by \eqref{intro13}, and if $\bar{z}^\mu$ has enough regularity to guarantee that all terms in \eqref{intro14} are meaningful, we obtain
\begin{equation}
\label{intro16}
V_\mu(x)=I_{-\infty}(\bar{\varphi}_\mu)+J^\mu_{-\infty}(\bar{z}^\mu)\geq V(x)+J^\mu_{-\infty}(\bar{z}^\mu).\end{equation}
In the same way, if $\bar{\varphi}$ is a minimizer for $V(x)$ and is regular enough, then
\begin{equation}
\label{intro17}
V_\mu(x)\leq I^\mu_{-\infty}(\bar{\varphi},\partial \bar{\varphi}/\partial t)=V(x)+J^\mu_{-\infty}((\bar{\varphi},\partial \bar{\varphi}/\partial t)).\end{equation}
If we could prove that
\begin{equation}
\label{intro18}
\liminf_{\mu\to 0}J^\mu_{-\infty}(\bar{z}^\mu)=\limsup_{\mu\to 0}J^\mu_{-\infty}((\bar{\varphi},\partial \bar{\varphi}/\partial t))=0,\end{equation}
from \eqref{intro16} and \eqref{intro17} we could conclude that \eqref{intro10} holds true. But unfortunately, neither $\bar{z}^\mu$ nor $\bar{\varphi}$ have the required regularity to justify \eqref{intro18}. Thus, we have to proceed with suitable approximations, which, among other things,  require us to prove the continuity of the mappings $V_\mu:D((-\Delta)^{1/2}Q^{-1})\to \reals$, uniformly with respect to  $\mu \in\,(0,1]$.

\medskip

In the second part of the paper we want to apply \eqref{intro10} to the study of the exit time and of the exit place of $u^\mu_\e$ from a given domain in $L^2(D)$ . For any open and bounded domain $G\subset L^2(D)$, containing the asymptotically stable equilibrium  $0$, and for any $z_0 \in\,G \times H^{-1}(D)$ we define the exit time
\[\tau^{\mu,\e}_{z_0}:=\inf\,\le\{t\geq 0\,:\,u^{\mu}_{\e,z_0}(t) \in\,\partial G\,\r\}.\]
Our first goal is to show that, for fixed $\mu>0$ and $z_0 \in\,G$,
\begin{equation}
\label{intro22}
\lim_{\e\to 0}\,\e\log\,\E\tau^{\mu,\e}_{z_0}=\inf_{x \in\,\partial G} V_\mu(x),
\end{equation}
and
\begin{equation} \label{intro23}
  \lim_{\e \to 0}\,\e\log\,(\tau^{\mu,\e}_{z_0}) = \inf_{x \in\,\partial G} V_\mu(x),\ \ \  \text{ in probability}.
\end{equation}
We also want to prove that if $N \subset \partial G$ has the property that $\ds{\inf_{x \in N} V_\mu(x) > \inf_{x \in \partial G} V_\mu(x) }$, then
\begin{equation} \label{intro24}
  \lim_{\e\to0}\, \Pro \left(u^\mu_{\e,z_0}(\tau^{\mu,\epsilon}_{z_0}) \in N  \right) =0.
\end{equation}

We would like to stress that the method we are using here in our infinite dimensional setting has several considerable differences compared to the classical finite dimensional argument developed in \cite{fw} (see also \cite{dz}).  The most fundamental difference between the two settings is that, unlike in the finite dimensional case,  in the infinite dimensional case the quasi-potentials $V_\mu$ are not continuous in $L^2(D)$. Nevertheless, we show here that the lower-semi-continuity of $V_\mu$ in $L^2(D)$ along with a convex type regularity assumption for the domain  $G$ are sufficient to prove our results. Another important difference is that $u^\mu_\e$ is not a Markov process, but the pair $(u^\mu_\e, \partial u^\mu_\e/\partial t)$ in the phase space $\H$ is.  For this reason, the exit time problem should be considered as the exit from the cylinder $G \times H^{-1} \subset \H$.  But, unfortunately, this is an unbounded domain, and as we show in section 3, the unperturbed trajectories are not uniformly attracted to zero from this cylinder. The methods we use to prove the exit time and exit place results should be applicable to most stochastic equations with second-order time derivatives.

In a similar manner, one can show that if 
\[\tau^\epsilon_{u_0} = \inf\{ t>0: u_\e(t) \not \in G\}\] is the exit time from $G$ for the solution of \eqref{semilinear-heat-eq-intro}, and $V(x)$ is the quasipotential associated with this system, the exit time and exit place results for the first-order system are analogous to \eqref{intro22}, \eqref{intro23}, and \eqref{intro24}.

As a consequence of \eqref{intro20}, in the gradient case, \eqref{intro22}, and \eqref{intro23} imply that,
for any fixed $\mu>0$, the exit time and exit place asymptotics of \eqref{semilinear-wave-eq-intro} match those of \eqref{semilinear-heat-eq-intro}. In particular, for any $\mu>0$
\begin{equation}
\label{i.1}\lim_{\e\to 0}\,\e\log\,\E\tau^{\mu,\e}_{z_0}=\inf_{x \in\,\partial G} V(x) = \lim_{\e \to 0} \e \log\,\E \tau^\epsilon_{u_0},\end{equation}
and
\begin{equation}
\label{i.2}
\lim_{\e \to 0} \e \log \tau^{\mu,\e}_{z_0} = \inf_{x \in\,\partial G} V(x) = \lim_{\e \to 0} \e \log\, \tau^\e_{u_0},\ \ \ \  \text{ in probability}.\end{equation}
Additionally, if there exists a unique $x^* \in \partial G$ such that $V(x^*) = \inf_{x \in \partial G} V(x)$, \eqref{intro24} implies that
\[\lim_{\epsilon \to 0} u^\mu_\epsilon(\tau^{\mu,\epsilon}) = x^* = \lim_{\e \to 0} u_\e(\tau^\e),\ \ \ \ \ \text{ in probability}.\]

In the general non-gradient case, we cannot have \eqref{i.1} and \eqref{i.2}. Nevertheless,  in view of \eqref{intro10}, the exit time and exit place asymptotics of \eqref{semilinear-wave-eq-intro} can be approximated by $V$. Namely
 \[\lim_{\mu \to 0}\lim_{\e\to 0}\,\e\log\,\E\tau^{\mu,\e}_{z_0}=\inf_{x \in\,\partial G} V(x) = \lim_{\e \to 0} \e \log\,\E \tau^\epsilon_{u_0},\]
and 
 \[\lim_{\mu \to 0}\lim_{\e \to 0} \e \log \tau^{\mu,\e}_{z_0} = \inf_{x \in\,\partial G} V(x) = \lim_{\e \to 0} \e \log\, \tau^\e_{u_0},\ \ \ \  \text{ in probability}.\]
Furthermore, if there exists a unique $x^* \in \partial G$ such that $V(x^*) = \inf_{x  \in \partial G} V(x)$, then
\[ \lim_{\mu \to 0} \lim_{\e \to 0} u^\mu_\epsilon(\tau^{\mu,\epsilon}) = x^* = \lim_{\e \to 0} u_\e(\tau^\e),\ \ \ \  \text{ in probability}.\]

\section{Preliminaries and assumptions}

Let $D$ be an open, bounded, regular domain in  $\mathbb{R}^d$, with $d\ge 1$ and let $H $ denote the Hilbert space $L^2(D)$.  In what follows, we shall denote by $A$  the realization in $H$ of the Laplace operator, endowed with Dirichlet boundary conditions, and we shall denote by   $\{e_k\}_{k \in\,\nat}$ and $\{-\a_k\}_{k \in\,\nat}$  the corresponding sequence of eigenfunctions and eigenvalues, with $0<\a_1\leq \a_k\leq \a_{k+1}$, for any $k \in\,\nat$.  Here, we assume that the domain $D$ is regular enough so that
\begin{equation}
\label{m18}
\a_k\sim k^{2/d},\ \ \ \ k \in\,\nat.
\end{equation}
For any $\d \in\,\reals$, we shall denote by $H^\d$ the completion of $C^\infty_0(D)$ with respect to the norm
\[|x|_{H^\d}^2=\sum_{k=1}^{+\infty} \a_k^{\d}\le<x,e_k\r>_H^2=\sum_{k=1}^\infty \a_k^\d x_i^2\]
$H^\d$ is a Hilbert space, endowed with the scalar product
\[\le<x,y\r>_{ H^\d}=\sum_{k=1}^{+\infty}\a_k^\d x_k y_k,\ \ \ x, y \in\,H^\d(D).\]
Finally, we shall denote by $\mathcal{H}_\d$ the Hilbert space $H^\d\times H^{\d-1}$ and in the case $\d=0$ we shall set $\mathcal{H}_0=\mathcal{H}$. Moreover, we shall denote
\[\Pi_1:\mathcal{H}_\d\to H^\d,\ \ (u,v)\mapsto u,\ \ \ \ \Pi_2:\mathcal{H}_\d\to H^{\d-1},\ \ (u,v)\mapsto v.\]
Sometimes, for the sake of simplicity, we will denote for any $\mu>0$ and $\d \in\,\mathbb{R}$
\begin{equation}
\label{m-1000}
I_\mu(u,v)=(u,\sqrt{\mu} v),\ \ \ \ (u,v) \in\,\mathcal{H}_\d.
\end{equation}

The stochastic perturbation  is given  by a cylindrical Wiener process $w^{Q}(t,\xi)$, for $t\geq 0$ and $\xi \in\,{\cal O}$, which
is assumed to be white in time and colored in space, in the case
of space dimension $d>1$. Formally, it  is defined as the infinite sum
\begin{equation}
\label{noise3}
w^{Q}(t,\xi)=\sum_{k=1}^{+\infty} Q e_{k}(\xi)\,\beta_{k}(t),
\end{equation} where
 $\{e_{k}\}_{k \in\,\nat}$ is the
complete orthonormal basis in $L^2(D)$ which diagonalizes $A$ and  $\{\beta_{k}(t)\}_{k
\in\,\nat}$ is a sequence of mutually independent standard
Brownian motions defined on the same complete stochastic basis
$(\Omega,\mathcal{F}, \mathcal{F}_t, \mathbb{P})$.
\begin{Hypothesis}
\label{H1}
The linear operator $Q$ is bounded in $H$ and diagonal with respect to the basis $\{e_k\}_{k \in\,\nat}$ which diagonalizes $A$. Moreover, if $\{\la_k\}_{k \in\,\nat}$ is the corresponding sequence of eigenvalues, we have
\begin{equation}
\label{m1}
\frac 1c \a_k^{-\beta}\leq \la_k\leq c\,\a_k^{-\beta},\ \ \ \ k \in\,\nat,
\end{equation}
for some $c>0$ and $\beta>(d-2)/4$.
\end{Hypothesis}
\begin{Remark}
{\em \begin{enumerate}
\item If $d=1$, according to Hypothesis \ref{H1} we can consider space-time white noise  ($Q=I$).
\item  Thanks to \eqref{m18}, condition \eqref{m1} implies that if $d\geq 2$, then there exists $\gamma<2d/(d-2)$ such that
\[\sum_{k=1}^\infty \la_k^\gamma<\infty.\]
Moreover
\[\sum_{k=1}^\infty \frac{\la_k^2}{\a_k}<\infty.\]
\item As a consequence of \eqref{m1}, for any $\d \in\,\reals$
\[D((-A)^{\d/2}Q^{-1})=H^{\d+2\beta}\]
and there exists $c_\d>0$ such that for any $x \in\,H^{\d+2\beta}$
\[\frac 1{c_\d}\,|(-A)^{\d/2}Q^{-1}x|_H\leq |x|_{\d+2\beta}\leq c_\d\,|(-A)^{\d/2}Q^{-1}x|_H\]
\end{enumerate}}
\end{Remark}

Concerning the nonlinearity $B$, we shall assume the following conditions.

\begin{Hypothesis}
\label{H2}
For any $\d \in\,[0,1+2\beta]$, the mapping $B:H^\d\to H^\d$ is Lipschitz continuous, with
\[ [B]_{\tiny{\text{{\em Lip}}}(H^{\delta})}=:\gamma_\d<\a_1.\]
Moreover $B(0)=0$.
We also assume that $B$ is differentiable in the space $H^{2\beta}$, and that $\sup_{z \in \H}\|DB(z)\|_{L(\H)} = \gamma_{2\beta}$.
\end{Hypothesis}

\begin{Remark}
{\em \begin{enumerate}
\item   The assumption that $B$ is differentiable is made for convenience to simplify the proof of lower bounds in Theorem \ref{t.82}. We believe that by approximating the Lipschitz continuous $B$ with a sequence of differentiable functions whose $C^2$ semi-norm is controlled by the Lipschitz semi-norm of $B$, the results proved in Theorem \ref{t.82} should remain true. 
\item If we define for any $x \in\,H$
\[B(x)(\xi)=b(\xi,x(\xi)),\ \ \ \xi \in\,D,\]
and we assume that $b(\xi,\cdot) \in\,C^{2k}(\reals)$, for $k \in\,[\beta+\d/2-5/4,\beta+\d/2-1/4]$, and
\[\frac{\partial^j b}{\partial \si^j}(\xi,\si)_{|_{\si=0}}=0,\ \ \ \xi \in\,\overline{D},\]
then $B$ maps $H^{\d}$ into itself, for any $\d \in\,[0,1+2\beta]$.
The Lipschitz continuity of $B$ in $H^{\d}$ and the bound on the Lipschitz norm, are satisfied if the derivatives of $b(\xi,\cdot)$ are small enough.
\end{enumerate}
}
\end{Remark}

With these notations, equation \eqref{semilinear-heat-eq-intro} can be written as the following abstract evolution equation in $H$
\begin{equation}
\label{abstract-heat}
du_\e(t)=\le[A u_\e(t)+B(u_\e(t))\r]\,dt+\sqrt{\e}\,dw^Q(t),\ \ \ \ u(0)=u_0.
\end{equation}

\begin{Definition}
A predictable process $u_\e \in\,L^2(\Omega;C([0,T];H))$ is a {\em mild solution} to equation \eqref{abstract-heat} if
\[u_\e(t)=e^{tA}u_0+\int_0^t e^{(t-s)A}B(u_\e(s))\,ds+\sqrt{\e}\,\int_0^t e^{(t-s)A}dw^Q(s).\]
\end{Definition}

Now, for each $\mu>0$ and $\d \in\,\reals$ we define $\Amu: D(\Amu)\subset \H_\d \to \H_\d$ by setting
\begin{equation}
 \Amu (u,v) = \left(-v, \frac{1}{\mu}Au - \frac{1}{\mu} v \right),\ \ \         (u,v)  \in\,D(A_\mu)=\mathcal{H}_{1+\d},
\end{equation}
and we denote by  $\Smu(t)$  the semigroup on $\mathcal {H}_\d$ generated by $\Amu$.  In \cite[Proposition 2.4]{smolu2}, it is proved that for each $\mu>0$ there exist $\omega_\mu>0$ and $ M_\mu>0$ such that
\begin{equation}
\label{m6}
\|S_\mu(t)\|_{\mathcal{L}(\mathcal{H})}\leq M_\mu\,e^{-\omega_\mu t},\ \ \ \ t\geq 0.
\end{equation}
Notice that, since for any $\d \in\,\reals$ and $(u,v) \in\,\H_\d$
\[\le((-A)^\d \Pi_1S_\mu(t)(u,v),(-A)^\d \Pi_2S_\mu(t)(u,v)\r)=S_\mu(t)((-A)^\d u,(-A)^\d v),\ \ \ t\geq 0,\]
\eqref{m6} implies that for any $\d \in\,\mathbb{R}$
\begin{equation}
\label{m17}
\|S_\mu(t)\|_{\mathcal{L}(\mathcal{H}_\d)}\leq M_\mu\,e^{-\omega_\mu t},\ \ \ \ t\geq 0.
\end{equation}

Next, for any $\mu>0$ we denote
\[B_\mu(u,v)=\frac 1\mu(0,B(u)),\ \ \ \ (u,v) \in\,\mathcal{H},\]
and
\[Q_\mu u=\frac 1\mu(0,Qu),\ \ \ \ u \in\,H.\]
With these notations, equation  \eqref{semilinear-wave-eq-intro} can be written as the following abstract evolution equation in the space $
\mathcal{H}$
\begin{equation}
\label{abstract}
dz(t)=\le[A_\mu z(t)+B_\mu(z(t))\r]\,dt+\sqrt{\e}\,Q_\mu\,dw(t),\ \ \ \ z(0)=(u_0,v_0).
\end{equation}

\begin{Definition}
\label{def2.4}
A predictable process $u^\mu_\e$
is a {\em mild solution} of \eqref{abstract} if
\[u^\mu_\e \in\,L^2(\Omega;C([0,T];H),\ \ \ \  v^\mu_\e=:\frac{\partial\, u^\mu_\e}{\partial t} \in\,L^2(\Omega;C([0,T];H^{-1}),\]
for any $T>0$, and
\begin{equation}
\label{semilinear-wave-mild-sol}
  z^{\mu}_{ \epsilon}(t) = \Smu(t)z(0) + \int_0^t \Smu(t-s) B_\mu(z^\mu_\e(s)) ds + \sqrt{\epsilon} \int_0^t \Smu(t-s)\Qmu dw(s),
\end{equation}
where  $z(0)=(u_0,v_0)$ and
$z^\mu_\e= \left(u^\mu_\e, v^\mu_\e \right)$.

\end{Definition}

In view of Hypothesis \ref{H1} and of the fact that $B:H\to H$ is Lipschitz continuous, for any $\mu>0$ and any initial condition $z_0=(u_0,v_0) \in\,\mathcal{H}$, there exists a unique mild solution $u^\mu_\e $ for equation \eqref{semilinear-wave-eq-intro},   (for a proof see e.g. \cite{smolu2}).
In \cite[Theorem 4.6]{smolu2} we have proved that for any fixed $\e>0$ and $T>0$ the solution $u^\mu_\e $ of equation \eqref{semilinear-wave-eq-intro} converges in $C([0,T];H)$, in probability sense,  to the solution $u_\e$ of equation \eqref{semilinear-heat-eq-intro}, as $\mu\downarrow 0$. Namely, for any $\eta>0$
\[\lim_{\mu\to 0}\,\Pro\le(\sup_{t \in\,[0,T]}|u^\mu_\e(t)-u_\e(t)|_H>\eta\r)=0.\]

\section{The unperturbed equation}

We consider here equation \eqref{abstract}, for $\e=0$. Namely,
\begin{equation}
\label{m-fine202}
\frac{dz}{dt}(t)=A_\mu z(t)+B_\mu(z(t)),\ \ \ \ z(0)=z_0=(u_0,v_0).
\end{equation}
The solution to \eqref{m-fine202} will be denoted by $z^\mu_{z_0}(t)$. We recall here that  $\gamma_0$ denotes the Lipschitz constant of $B$ in $H$ (see  Hypothesis \ref{H2}).

\begin{Lemma} \label{X-sup-bounded-lemma}
  If $\mu < (\alpha_1 - \gamma_0)\gamma_0^{-2}$, there exists a constant $c_1(\mu)>0$ such that
  \begin{equation} \label{X-L2-sup-bounded}
    \sup_{t \ge 0}\left|z^{\mu}_{z_0}(t) \right|_\H +  \left| z^{\mu}_{z_0} \right|_{L^2((0,+\infty);\H)} \leq c_1(\mu) |z_0|_\H,\ \ \ \ z_0 \in\,\H.
  \end{equation}
\end{Lemma}

\begin{proof}
If $\varphi(t) = \Pi_1 z^{\mu}_{z_0}(t)$ then
  \begin{equation}  \label{unperturbed-diff-eq}
    \mu \frac{\partial^2 \varphi}{\partial t^2}(t) +  \frac{\partial \varphi}{\partial t}(t) = A \varphi(t) + B(\varphi(t)).
  \end{equation}
  By taking the inner product of \eqref{unperturbed-diff-eq} with $\frac{\partial \varphi}{\partial t}$ in $H^{-1}$, and by using the Lipschitz continuity of $B$ in $H$, we see that
  \begin{equation} \label{deriv-of-varphi-product}
    \mu  \frac{d}{dt} \left| \frac{\partial \varphi}{\partial t}(t) \right|_{H^{-1}}^2 + 2 \left| \frac{\partial \varphi}{\partial t}(t) \right|_{H^{-1}}^2 \leq- \frac{d}{dt} \left| \varphi(t) \right|_H^2 + \left| \frac{\partial \varphi}{\partial t}(t) \right|_{H^{-1}}^2 + \frac{\gamma_0^2}{\alpha_1} \left|\varphi(t) \right|_H^2.
  \end{equation}
By integrating this expression in time, we see that
  \begin{equation} \label{deriv-of-varphi-prod-int}
    \mu \left| \frac{\partial \varphi}{\partial t}(t) \right|_{H^{-1}}^2 + \left| \varphi(t) \right|_H^2 + \int_0^t \left|\frac{\partial \varphi}{\partial s}(s) \right|_{H^{-1}}^2 ds \leq\mu |v_0|_{H^{-1}}^2 + |u_0|_{H}^2 + \frac{\gamma_0^2}{\alpha_1} \int_0^t \left| \varphi(s) \right|_H^2 ds.
  \end{equation}

Next, by taking the inner product of \eqref{unperturbed-diff-eq} with $\varphi(t)$ in $H^{-1}$, since
  \begin{equation*}
    \left< \frac{\partial^2 \varphi}{\partial t^2}(t), \varphi(t) \right>_{H^{-1}} = \frac 12 \frac{d^2}{dt^2} \left|\varphi(t)\r|^2_{H^{-1}}
    - \left|\frac{\partial \varphi}{\partial t}(t) \right|_{H^{-1}}^2,
  \end{equation*}
   we have
\[\mu \frac{d^2}{dt^2}\left|\varphi(t)\r|^2_{H^{-1}} + \frac{d}{dt} |\varphi(t)|_{H^{-1}}^2 \leq- 2|\varphi(t)|_H^2 +  \frac{2\gamma_0}{\alpha_1} |\varphi(t)|_H^2 + 2\mu \left|\frac{\partial \varphi}{\partial t}(t)\right|_{H^{-1}}^2.
  \]
  By \eqref{deriv-of-varphi-product}, this yields
  \begin{equation}
  \label{m-fine203}\begin{array}{l}
  \ds{\mu \frac{d^2}{dt^2}\left|\varphi(t)\r|^2_{H^{-1}}+ \frac{d}{dt} |\varphi(t)(t)|_{H^{-1}}^2 } \leq\\
  \vs
  \ds{- 2|\varphi(t)|_H^2 +  \frac{2\gamma_0}{\alpha_1} |\varphi(t)|_H^2 - 2 \mu^2 \frac{d}{dt} \left| \frac{\partial \varphi}{\partial t}(t) \right|_{H^{-1}}^2- 2\mu \frac{d}{dt} |\varphi(t)|_H^2 + \frac{2\gamma_0^2 \mu}{\alpha_1} |\varphi(t)|_H^2.}
  \end{array}
  \end{equation}
Now,  if $\mu< (\alpha_1 - \gamma_0)\gamma_0^{-2}$, it follows
  \begin{equation*}
  \rho_\mu := 2 - \frac{2\gamma_0}{\alpha_1} - \frac{2\mu \gamma_0^2}{\alpha_1}>0.
  \end{equation*}
Then, by integrating both sides in \eqref{m-fine203}, we see
  \begin{equation}
  \label{m-fine220}
  \begin{array}{l}
 \ds{\mu \frac{d}{dt} \left| \varphi(t) \right|_{H^{-1}}^2 + \left| \varphi(t) \right|_{H^{-1}}^2 +  \rho_\mu \int_0^t \left| \varphi(s) \right|_H^2 ds} \\
 \vs
 \ds{\leq 2 \mu \left< v_0,u_0 \right>_{H^{-1}} + |u_0|_{H^{-1}}^2 + 2\mu^2 |v_0|_{H^{-1}}^2 + 2\mu |u_0|_H^2,}
 \end{array}
  \end{equation}
  and this implies that
  \begin{equation} \label{varphi-L2-bound}
    \int_0^\infty |\varphi(t)|_H^2 ds \leq\frac{1}{\rho_\mu} \left(2 \mu \left< v_0,u_0 \right>_{H^{-1}} + |u_0|_{H^{-1}}^2 + 2\mu^2 |v_0|_{H^{-1}}^2 + 2\mu |u_0|_H^2 \right).
  \end{equation}
  Actually, if there exists $t_0>0$ and $\d>0$ such that
  \[   \int_0^{t_0} |\varphi(t)|_H^2 ds>\frac{1}{\rho_\mu} \left(2 \mu \left< v_0,u_0 \right>_{H^{-1}} + |u_0|_{H^{-1}}^2 + 2\mu^2 |v_0|_{H^{-1}}^2 + 2\mu |u_0|_H^2 \right)+\d,\]
  then, in view of \eqref{m-fine220}, for any $t>t_0$
  \[\mu \frac{d}{dt} \left| \varphi(t) \right|_{H^{-1}}^2 <-\d.\]
  This would imply that for any $t>t_0$
  \[|\varphi(t)|_{H^{-1}}^2<|\varphi(t_0)|_{H^{-1}}^2-(t-t_0)\d,\]
  which is impossible.

  We conclude the proof by combining \eqref{deriv-of-varphi-prod-int} and \eqref{varphi-L2-bound}, to see that
  \begin{equation*}
    \mu \left| \frac{\partial \varphi}{\partial t}(t) \right|_{H^{-1}}^2 + \left| \varphi(t) \right|_H^2 + \int_0^t \left|\frac{\partial \varphi}{\partial s}(s) \right|_{H^{-1}}^2 ds + \int_0^t \left| \varphi(s) \right|_H ds \leq c |z_0|_\H^2.
  \end{equation*}

\end{proof}

\begin{Lemma} \label{unperturbed-system-conv-to-zero-lem}
Assume  $\mu < (\alpha_1 - \gamma_0)\gamma_0^{-2}$, then for any $R>0$,
  \begin{equation}
    \lim_{t \to +\infty} \sup_{|z_0|_\H \leq R} \left| z^{\mu}_{z_0}(t) \right|_\H =0.
  \end{equation}
\end{Lemma}

\begin{proof}
    Let us fix $R,\rho>0$ and for any $\mu>0$ let us define
\[T = \frac{(c_1(\mu))^4 R^2}{\rho^2}.\]
Let $|z_0|_\H \leq R$. Since    \begin{equation*}
     \left|z^{\mu}_{z_0} \right|_{L^2((0,T);\H)} \ge \sqrt{T} \min_{s\leq T}\,|z^{\mu}_{z_0}(s)|_\H,
    \end{equation*}
    according to \eqref{X-L2-sup-bounded}
     there must exists $t_0<T$ such that
     \[|z^{\mu}_{z_0}(t_0)|_\H \leq \frac{\rho}{c_1(\mu)}.\]
By using again  \eqref{X-L2-sup-bounded}, this implies
    \begin{equation*}
      \sup_{t \ge T} \left|z^{\mu}_{z_0}(t) \right|_{\H}=      \sup_{t \ge T}  \left|z^{\mu}_{z^\mu_{z_0}(t_0)}(t-t_0) \right|_{\H}\leq \rho.
    \end{equation*}
    Notice that $T$ is independent of our choice of $z_0$ so we can conclude that
    \begin{equation*}
      \sup_{t \ge T} \sup_{|z_0|_\H \leq R} \left|z^\mu_{z_0} (t) \right|_\H \leq \rho.
    \end{equation*}
\end{proof}

Now that we have shown that the unperturbed system is uniformly attracted to $0$ from any bounded set in $\H$, we show that if the initial velocity is large enough, $\Pi_1 z^{\mu}_{z_0}$ will leave any bounded set.

\begin{Lemma}
  For any $\mu>0$ and $t>0$, there exists $c_2(\mu,t)>0$ such that
  \begin{equation} \label{large-velocity-causes-exit-eq}
   \sup_{s \leq t} \left| \Pi_1 \Smu(s) \left(0,v_0 \right) \right|_H \geq c_2(\mu,t) \left|v_0 \right|_{H^{-1}},\ \ \ \  v_0 \in H^{-1}.
  \end{equation}
\end{Lemma}

\begin{proof}
  Let $\varphi(t) = \Pi_1 \Smu(t) (0,v_0)$.  Then
  \begin{equation*}
    \mu \frac{\partial^2 \varphi}{\partial t^2}(t) + \frac{\partial \varphi}{\partial t}(t) = A \varphi(t),\ \ \ \ \varphi(0) = 0,\ \ \frac{\partial \varphi}{\partial t}(0)=v_0.
  \end{equation*}
By taking the inner product of this equation with $\frac{\partial \varphi}{\partial t}(t)$ in $H^{-1}$, we see that
  \begin{equation*}
    \mu \frac{d}{dt} \left| \frac{\partial \varphi}{\partial t}(t) \right|_{H^{-1}}^2 + 2 \left| \frac{\partial \varphi}{\partial t}(t) \right|_{H^{-1}}^2
    = -\frac{d}{dt} \left| \varphi(t) \right|_H^2.
  \end{equation*}
Therefore, by standard calculations,
  \begin{equation*}
\begin{array}{l}
\ds{    \left| \frac{\partial \varphi}{\partial t}(t) \right|_{H^{-1}}^2 = e^{-\frac{2t}{\mu}} |v|_{H^{-1}}^2 - \frac{1}{\mu} \int_0^t e^{-\frac{2(t-s)}{\mu}} \frac{d}{ds} \left|\varphi(s) \right|_H^2ds}\\
\vs
\ds{ = e^{-\frac{2t}{\mu}} |v_0|_{H^{-1}}^2 - \frac{1}{\mu}  |\varphi(t)|_H^2 + \frac{2}{\mu^2} \int_0^t e^{-\frac{2(t-s)}{\mu}} \left|\varphi(s) \right|_H^2 ds,}
\end{array}
  \end{equation*}
so that
  \begin{equation} \label{deriv-of-varphi-upper-bound}
    \left| \frac{\partial \varphi}{\partial t}(t) \right|_{H^{-1}}^2 \leq e^{-\frac{2t}{\mu}} |v_0|_{H^{-1}}^2 + \frac{1}{\mu} \sup_{s \leq t} |\varphi(s)|_H^2.
  \end{equation}

  Next, since
\begin{equation} \label{linear-eq-energy-method}  \begin{array}{l}
\ds{\frac{d}{dt} \left| \varphi(t) + \mu\frac{\partial \varphi}{\partial t}(t) \right|_{H^{-1}}^2
= 2\left<\varphi(t) + \mu \frac{\partial \varphi}{\partial t}(t), \frac{\partial \varphi}{\partial t}(t) + \mu \frac{\partial^2 \varphi}{\partial t^2}(t) \right>_{H^{-1}}} \\
\vs
\ds{= 2\left< \varphi(t) + \mu \frac{\partial \varphi}{\partial t}(t), A \varphi(t) \right>_{H^{-1}}
= - 2|\varphi(t)|_H^2 - \mu \frac{d}{dt} \left| \varphi(t) \right|_{H},}
  \end{array}\end{equation}
if we integrate in time we get
 \begin{equation*}
\mu^2 |v_0|_{H^{-1}}^2=    \left| \varphi(t) + \mu \frac{\partial \varphi}{\partial t}(t) \right|_{H^{-1}}^2 + 2 \int_0^t |\varphi(s)|_H^2 ds + \mu |\varphi(t)|_H^2.
  \end{equation*}
  For any $a>0$ to be chosen later, we have
  \begin{equation*}
    \left| \varphi(t) + \mu \frac{\partial \varphi}{\partial t}(t) \right|_{H^{-1}}^2 \leq\left( 1+ a^{-1}\right) \frac{1}{\alpha_1} \left| \varphi(t) \right|_H^2 + \mu^2 (1+a) \left|\frac{\partial \varphi}{\partial t}(t)\right|_{H^{-1}}^2
  \end{equation*}
and therefore,
  \begin{equation*}
    \mu^2 |v_0|_{H^{-1}}^2 \leq\left( \mu + 2t + \left(1 + a^{-1} \right)\frac{1}{\alpha_1} \right) \sup_{s \leq t} |\varphi(s)|_H^2 + \mu^2 (1+a) \left| \frac{\partial \varphi}{\partial t}(t) \right|_{H^{-1}}^2
  \end{equation*}
Thanks to \eqref{deriv-of-varphi-upper-bound}, this yields
\[\mu^2 \left( 1 - (1+a) e^{-\frac{2t}{\mu}} \right) |v_0|_{H^{-1}}^2 \leq\left( \mu + 2t + \left(1 + a^{-1}  \right)\frac{1}{\alpha_k} + (1+a)\mu  \right) \sup_{s \leq t} \left| \varphi(s) \right|_H^2,
  \]
 and our conclusion follows with if we pick $a< e^{\frac{2t}{\mu}} - 1$.

\end{proof}

As a consequence of the previous lemma, we can conclude that the following lower bound estimate holds for the solution of \eqref{m-fine202}.

\begin{Lemma} \label{large-init-veloc-causes-exit-lem}
  For any $\mu>0$ and $t> 0$ there exists $c(\mu,t)>0$ such that
  \begin{equation}
    \sup_{s \leq t} \left|\Pi_1 z^{\mu}_{z_0}(s) \right|_H \ge c(\mu,t) |\Pi_2 z_0|_{H^{-1}},\ \ \ \ z_0 \in\,\H.
  \end{equation}
\end{Lemma}

\begin{proof}
Let $z_0 = (u_0,v_0)$. Since  \begin{equation*}
    \Pi_1 z^{\mu}_{z_0}(t) = \Pi_1 \Smu(t) (u_0,0) + \Pi_1 \Smu(t) (0,v_0) + \Pi_1 \int_0^t \Smu(t-s) B_\mu(z^{\mu}_{z_0}(s)) ds,
  \end{equation*}
from the Hypothesis \ref{H2} and \eqref{m6}, for any $s>0$
  \begin{equation*}
    |\Pi_1 \Smu(s)(0,v_0)|_H \leq \left(2M_\mu + \frac{\gamma_0 M_\mu}{\omega_\mu \mu}  \right) \sup_{r \leq s} \left|\Pi_1 z^{\mu}_{z_0}(r) \right|_H.
  \end{equation*}
  According to \eqref{large-velocity-causes-exit-eq}, this implies that for any $t>0$,
  \begin{equation*}
c_2(\mu,t) |v_0|_{H^{-1}} \leq \sup_{s \leq t} |\Pi_1 \Smu(t)(0,v_0)|_H \leq \left(2 M_\mu +  \frac{\gamma_0 M_\mu}{\omega_\mu \mu} \right) \sup_{s \leq t} \left|\Pi_1 z^{\mu}_{z_0}(s) \right|_H.
  \end{equation*}
  Therefore, the result follows with
  \[c(\mu,t) = c_1(\mu,t) \left( 2 M_\mu +  \frac{\gamma_0 M_\mu}{\omega_\mu \mu}\right)^{-1}.\]
\end{proof}

\section{The skeleton equation}

For any  $\mu>0$ and $s< t$ and for any $\psi \in\,L^2((s,t);H)$ we define
\[L^\mu_{s,t} \psi = \int_{s}^{t} \Smu(t -r) \Qmu \psi(r) dr.\]
Clearly $L^\mu_{s,t}$ is a continuous bounded linear operator from $L^2([s,t];H)$ into $\mathcal{H}$. If we define the pseudo-inverse of $L^\mu_{s,t}$ as
\[(L^\mu_{s,t})^{-1}(x)=\arg\!\min\,\le\{|(L^\mu_{s,t})^{-1}(\{x\})|_{L^2([s,t];H)}\r\},\ \ \ \ x \in\,\Im(L^\mu_{s,t}),\]
  we have the following bounds.

\begin{Theorem}  \label{L-inverse-thm}
For any $\mu>0$ and $s<t$, it holds
\begin{equation}
\label{m11}
\left | \left(L^\mu_{s,t} \right)^{-1} z \right |_{L^2((s,t);H)} =
    \sqrt{2}\left |   (C_\mu - \Smu(t-s) C_\mu S^{\star}_\mu(t-s))^{-1/2} z \right |_{\mathcal{H}},\ \ \ \ z \in\,\Im(L^\mu_{s,t}),
    \end{equation}
      where
  \begin{equation} \label{D_mu-def}
    C_\mu (u,v) = \left(Q^2(-A)^{-1} u, \frac{1}{\mu} Q^2 (-A)^{-1} v \right),\ \ \ \ (u,v) \in\,\mathcal{H}.
  \end{equation}
    Moreover,  for every $\mu>0$ there exists $T_\mu>0$ such that
  \begin{equation}
  \label{m12}
    \Im(L^\mu_{s,t}) = \Im((C_\mu)^{1/2})=\H_{1+2\beta},\ \ \ \ t-s \geq T_\mu,
  \end{equation}
  and
  \begin{equation} \label{L-mu-inverse-bound}
    \left | (L^\mu_{s,t})^{-1} z  \right |_{L^2((s,t);H)} \leq c(\mu, t-s) \left |z \right |_{\H_{1+2\beta}},\ \ \ z \in\,\H_{1+2\beta},
  \end{equation}
  for some constant $c(\mu,r)>0$, with $r\geq T_\mu$.

\end{Theorem}

\begin{proof}
It is immediate to check that for any $z \in\,\mathcal{H}$
  \begin{align} \label{L-mu-adj-initial-calculation}
    &\left |(L^\mu_{s,t})^\star z \right |_{L^2((s,t);H)}^2
    = \frac{1}{\mu^2}\int_{0}^{t-s} \left | Q(-A)^{-1} \Pi_2 S_\mu^\star(r) z \right |_H^2 dr.
  \end{align}
Now, if we expand $S^\star_\mu(t)(u,v)$ in Fourier series, we have (see \cite[Proposition 2.3]{smolu2})
\[    S^\star_\mu(t) (u,v) = \sum_{k=1}^\infty \left(\hat{f}_k^\mu(t) e_k, \hat{g}_k^\mu(t) e_k \right),\]
where $\hat{f}_k^\mu$ and $\hat{g}_k^\mu$ solve the system
  \begin{equation}
\le\{\begin{array}{ll}
\ds{      \mu(\hat{f}_k^\mu)'(t) = -\hat{g}_k^\mu(t),}  &  \ds{\hat{f}_k^\mu(0) = u_k,}\\
&  \vs
\ds{ \mu (\hat{g}_k^\mu)'(t) = \mu \alpha_k \hat{f}_k^\mu(t) - \hat{g}_k^\mu(t),}   &   \ds{\hat{g}_k^\mu(0) = v_k.}
\end{array}\r.
  \end{equation}
In particular,
  \begin{align}
    |\hat{g}_k^\mu(t)|^2 = - \frac{\mu^2 \alpha_k}{2} \frac{d}{dt} |\hat{f}_k^\mu(t)|^2 - \frac{\mu}{2} \frac{d}{dt} |\hat{g}_k^\mu(t)|^2
  \end{align}
Due to \eqref{L-mu-adj-initial-calculation}, we get
\begin{equation}
 \label{L-mu-adj-norm-squared}
 \begin{array}{l}
 \ds{\left |(L^\mu_{s,t})^\star z \right |_{L^2([s,t];H)}^2=\frac 12 \sum_{k=1}^\infty \int_{0}^{t-s} \left( - \frac{\lambda_k^2}{\alpha_k} \frac{d}{dr} |\hat{f}_k^\mu(r)|^2 - \frac{\lambda_k^2}{ \mu \alpha_k^2} \frac{d}{dr} |\hat{g}_k^\mu(r)|^2 \right) dr}\\
 \vs
 \ds{= \frac 12\sum_{k=1}^\infty \left( -\frac{\lambda_k^2}{\alpha_k} |\hat{f}_k^\mu(t-s)|^2 - \frac{\lambda_k^2}{\alpha_k^2 \mu} |\hat{g}_k^\mu(t-s)|^2 + \frac{\lambda_k^2}{\alpha_k} |u_k|^2 + \frac{\lambda_k^2}{\alpha_k^2 \mu} |v_k|^2 \right)} \\
 \vs
\ds{=\frac{1}{2} \left( | C_\mu^{1/2} z |_\H^2
                         -| C_\mu^{1/2} S_\mu^\star(t-s)z  |_\H^2 \right)= \frac{1}{2} \left<(C_\mu - \Smu(t-s) C_\mu S_\mu^\star(t-s)) z,z \right>_\H.}
  \end{array}
  \end{equation}
  This implies that
  \[\Im(L^\mu_{s,t})=\Im((C_\mu - \Smu(t-s) C_\mu S_\mu^\star(t-s))^{1/2},\]
  and \eqref{m11} follows.

  Next, in order to prove \eqref{m12}, we notice that
  \begin{equation}
  \label{m15}
   C_1^{1/2} S_\mu^\star(t)=S_\mu^\star(t) C_1^{1/2},\ \ \ \ t\geq 0, \end{equation}
  and that
  \[ \left( 1 \wedge \sqrt{\mu} \right) \le| C_\mu^{1/2} z \r|_\H \leq \le| C_1^{1/2} z \r|_\H \leq \left( 1 + \sqrt{\mu} \right) \left|C_\mu^{1/2} z \right|_\H
  \]
so that, due to \eqref{m6}, we have
  \[| C_\mu^{1/2} S_\mu^\star(t)z |_\H \leq c_\mu M_\mu e^{ -\omega_\mu t} | C_\mu^{1/2} z |_\H,\ \ \ \ \ t\geq 0.\]
According to \eqref{L-mu-adj-norm-squared}, this implies
\[    \left |(L^\mu_{s,t})^* z \right |_\H^2 = \frac 12 |C_\mu^{1/2} z |_\H^2 - \frac 12 | C_\mu^{1/2} \Smu(t-s) z |_\H\ge \frac 12 (1 -c_\mu^2\, M_\mu^2 e^{-2 \omega_\mu (t-s)}) |C_\mu^{1/2} z |_\H^2.\]
Therefore, if we pick $T_\mu>0$  large enough so that $c_\mu^2\,M_\mu e^{-\omega_\mu T_\mu}<1$, we obtain that
\[   \Im(L^\mu_{s,t}) = \Im((C_\mu)^{1/2}),\]
and
\[  \left | (L^\mu_{s,t})^{-1} z  \right |_{L^2((s,t);H)} \leq \sqrt{2}\,\le(1 - c_\mu^2\, M_\mu^2 e^{-2 \omega_\mu r}\r)^{-1/2}\,|(C_\mu)^{-1/2}z|_{\H}.\]
Now,
as for any $\mu>0$ we have
$\Im((C_\mu)^{1/2})=\H_{1+2\beta},$
and
\begin{equation}
\label{m45}
(1\wedge \mu)\,|z|_{\H_{1+2\beta}}\leq |(C_\mu)^{-1/2}z|_{\H}\leq (1+\mu)\,|z|_{\H_{1+2\beta}},
\end{equation}
\eqref{m12} and \eqref{L-mu-inverse-bound} follow immediately, with
\[c(\mu,r) = (1+\mu)\,\sqrt{2}\,\le(1 -c_\mu^2\,  M_\mu^2 e^{-2 \omega_\mu r}\r)^{-1/2}.\]

\end{proof}

\begin{Remark}
{\em \begin{enumerate}
\item In fact,  it is possible to show that $\Im(L^\mu_{s,t}) = \Im((C_\mu)^{1/2})$, for all $t-s>0$, by using the explicit representation of $S_\mu^\star(t)$.
\item From \eqref{m6} and \eqref{m11}, it easily follows that
\begin{equation}
\label{m41}
|(L^\mu_{-\infty,t})^{-1}z|_{L^2((-\infty,t);H)}=\sqrt{2}\,|C_\mu^{-1/2}z|_\H,\ \ \ z \in\,\text{Im}(L^\mu_{-\infty,t}).
\end{equation}
\end{enumerate}}
\end{Remark}

\begin{Lemma}  \label{wave-high-reg-conv-to-zero-lem}
Let us fix $\psi \in L^2((-\infty,0);H^{2\a})$, with $\a \in\,[0,1/2]$, and $\mu>0$ and let  $z^\mu_\psi \in\,C((-\infty,0);\H)$ solve the equation
\begin{equation}
\label{m16}
z^\mu_\psi(t)=\int_{-\infty}^t S_\mu(t-s) B_\mu(z^\mu_\psi(s))\,ds+\int_{-\infty}^tS_\mu(t-s) Q_\mu\psi(s)\,ds,\ \ \ t\in\,\reals.
\end{equation}
Then, if
\begin{equation}
\label{m31}\lim_{t\to-\infty}|z^\mu_\psi(t)|_\H=0,
\end{equation}
 we have $z^\mu_\psi \in\,C((-\infty,0);\H_{1+2(\a+\beta)})$ and
   \begin{equation}
   \label{m30}
      \lim_{t \to -\infty} \left |z^\mu_\psi(t) \right|_{\H_{1+2(\a+\beta)}} =0.
    \end{equation}
\end{Lemma}

\begin{proof}
According to \eqref{m17}, for any $\d>0$ we have
  \[
  \begin{array}{l}
  \ds{\left|\int_{-\infty}^t \Smu(t-s)  B_\mu(z^\mu_\psi(s))  ds \right |_{\H_\d} \leq  \frac{M_\mu}{\mu}\sup_{s \leq t} \left|B(\Pi_1 z^\mu_\psi(s)) \right|_{H^{\d-1}} \int_{-\infty}^t e^{-\omega_\mu(t-s)} ds}\\
  \vs
  \ds{ \leq \frac{M_\mu}{\mu\, \omega_\mu} \sup_{s \leq t} \left|B(\Pi_1z^\mu_\psi(s)) \right|_{H^{\d-1}}.}
    \end{array}
  \]
  Therefore, due to Hypothesis \ref{H2}, if we take $\d=1$
  \begin{equation}
   \label{limit-to-zero-bootstrap}
\left| \int_{-\infty}^t \Smu(t-s)  B_\mu(z^\mu_\psi(s))  ds \right |_{\H_1} \leq    \frac{M_\mu\gamma_0}{\mu\, \omega_\mu} \sup_{s \leq t} \left|\Pi_1 z^\mu_\psi(s) \right|_H.
\end{equation}
For the second term in \eqref{m16}, if $\psi \in\,L^2(-\infty,0;H^{2\a})$, then $Q_\mu \psi \in\,L^2((-\infty,0);\H_{1+2(\a+\beta)})$, with
\[|Q_\mu \psi|_{L^2((-\infty,t);\H_{1+2(\a+\beta)})}\leq \frac c\mu \,|\psi|_{L^2((-\infty,t);H^{2\a})},\ \ \ t\leq 0.\]
Due to \eqref{m17}, this yells
  \begin{equation} \label{control-term-goes-to-zero}
\left|  \int_{-\infty}^t \Smu(t-s) \Qmu \psi(s)  ds \right|_{\H_{1+2(\a+\beta)}}
    \leq  \frac{M_\mu}{\mu}\le(\int_0^{\infty}e^{-2\omega_\mu s}\,ds\r)^{1/2}| \psi|_{L^2((-\infty,t);H^{2\a})}.
  \end{equation}
Therefore, from \eqref{m16}, \eqref{limit-to-zero-bootstrap} and \eqref{control-term-goes-to-zero}, we get
\[\begin{array}{l}
\ds{\left | z^\mu_\psi(t) \right|_{\H_1}\leq c_\mu \le(\sup_{s \leq t} \left|\Pi_1 z^\mu_\psi(s) \right|_H+| \psi|_{L^2((-\infty,t);H^{2\a})}\r).}
\end{array}\]
In particular, we have $z^\mu_\psi \in\,L^\infty((-\infty,0);\H_{1})$ and
\[\lim_{t \to -\infty} \left|z^\mu_\psi(t) \right|_{\H_1} =0.
  \]
  Now, by repeating the same arguments, we can prove that for any $n \in\,\nat$, with $n\leq [1+2\beta]$, if
\[z^\mu_\psi \in\,L^\infty((-\infty,0);\H_{n}),\ \ \ \text{and}\ \ \ \lim_{t \to -\infty} \left|z^\mu_\psi(t) \right|_{\H_n} =0,\]
then
\[z^\mu_\psi \in\,L^\infty((-\infty,0);\H_{n+1}),\ \ \ \text{and}\ \ \ \lim_{t \to -\infty} \left|z^\mu_\psi(t) \right|_{\H_{n+1}}=0.\]
Since there exists $\bar{n} \in\,\nat$ such that $\H_{1+2(\a+\beta)}\supset \H_{\bar{n}}$, we can conclude that $z^\mu_\psi$ belongs to  $L^\infty((-\infty,0);\H_{1+2(\a+\beta)})$ and \eqref{m30} holds. Continuity follows easily, by standard arguments.
\end{proof}

\begin{Remark}
{\em \begin{enumerate}
\item From the previous lemma, we have that if $z^\mu_\psi \in\,C((-\infty,0);\H)$ solves equation \eqref{m16} and limit \eqref{m31} holds, then $z^\mu_\psi(t) \in\,\H_{1+2\beta}$, for any $t\leq 0$. In particular $z^\mu_\psi(0) \in\,\H_{1+2\beta}$.
\item In \cite[Lemma 3.5]{cerrok}, it has been proven that the same holds for
equation \eqref{abstract-heat}. Actually, if $\varphi_\psi \in\,C((-\infty,0);H)$ is the solution to
\[    \varphi_\psi(t) = \int_{-\infty}^t e^{(t-s)A} B(\varphi_\psi(s)) ds + \int_{-\infty}^t e^{(t-s)A} Q \psi(s) ds,
  \]
  for $\psi \in\,L^2((-\infty,0);H)$, and
  \[\lim_{t\to-\infty}|\varphi_\psi(t)|_H=0,\]
then $\varphi_\psi \in\,C((-\infty,0);H^{1+2\beta})$ and there exists a constant such that for all $t\leq 0$,
  \begin{equation}
    \left|\varphi_\psi(t) \right|_{H^{1+2\beta}} \leq c\, |\psi|_{L^2((-\infty,0;H)}.
  \end{equation}
  Moreover,
  \begin{equation}
  \label{m21}
  \lim_{t\to-\infty}|\varphi_\psi(t)|_{H^{1+2\beta}}=0.
  \end{equation}
  \end{enumerate}
 }
\end{Remark}

\begin{Lemma} \label{energy-est-lemma}
Let $\a \in\,[0,1/2]$ and let $\psi_1, \psi_2 \in\,L^2((-\infty,0);H^{2\a})$. In correspondence of each $\psi_i$, let $z^\mu_{\psi_i} \in\,C((-\infty,0);\H_{1+2(\a+\beta)})$ be a solution of equation \eqref{m16}, verifying
 \eqref{m31}.
Then,
$z^\mu_{\psi_i}\in\,L^2((-\infty,0);\H_{1+2(\a+\beta)})$, for $i=1,2$, and
there exist $\mu_0>0$  and $c>0$ such that for any $\mu\leq \mu_0$ and $\tau\leq 0$
\begin{equation}
\label{m20}
\begin{array}{l}
\ds{|z^\mu_{\psi_1}-z^\mu_{\psi_2}|_{L^2((-\infty,\tau);\H_{1+2(\a+\beta)})}^2 + \sup_{t\leq \tau} \left|I_\mu(z^\mu_{\psi_1}(t) - z^\mu_{\psi_2}(t)) \right|_{\H_{1+ 2(\alpha +\beta)}}^2 \leq c\,|\psi_1-\psi_2|_{L^2((-\infty,\tau);H^{2\a})}^2,}
\end{array}
\end{equation}
where $I_\mu$ is defined in \eqref{m-1000}.

\end{Lemma}

\begin{proof}
If we define
\[u(t)=(-A)^{\a+\beta}\Pi_1\le(z^\mu_{\psi_1}(t)- z^\mu_{\psi_2}(t)\r),\ \ \ t\leq 0,\]
and
\[\psi(t) = (-A)^{\a+\beta} Q (\psi_1 (t)- \psi_2(t) ),\ \ \ t\leq 0,\] we have
  \begin{equation}
  \label{m33}
  \begin{array}{l}
  \ds{\mu \frac{\partial^2 u}{\partial t^2}(t) +\frac{\partial u}{\partial t}(t)=A u(t)+ (-A)^{\a+\beta} \left( B(\Pi_1 z^\mu_{\psi_1}(t)) - B(\Pi_1 z^\mu_{\psi_2}(t)) \right)+ \psi(t).}
  \end{array}
  \end{equation}
According to Hypothesis \ref{H2},  $B:H^{2(\a+\beta)}\to H^{2(\a+\beta)}$ is Lipschitz-continuous, and  then
\[\begin{array}{l}
\ds{\le|(-A)^{\a+\beta} \left( B(\Pi_1 z^\mu_{\psi_1}(t)) - B(\Pi_1 z^\mu_{\psi_2}(t)) \right)\r|_H=| B(\Pi_1 z^\mu_{\psi_1}(t)) - B(\Pi_1 z^\mu_{\psi_2}(t))|_{H^{2(\a+\beta)}}}\\
\vs
\ds{\leq \gamma_{2(\a+\beta)}\,|\Pi_1( z^\mu_{\psi_1}(t))-z^\mu_{\psi_2}(t))|_{H^{2(\a+\beta)}}=\gamma_{2(\a+\beta)}\,|u(t)|_H.}
\end{array}\]
Therefore, by taking the scalar product of   both sides with $\partial u/\partial t$, we get
 \begin{equation}
 \label{m32}
 \begin{array}{l}
 \ds{
    \left|\frac{\partial u}{\partial t}(t) \right|_H^2  +\frac{\mu}{2} \frac{d}{dt} \left| \frac{\partial u}{\partial t}(t) \right|_H^2
      +\frac{1}{2} \frac{d}{dt} \left|(-A)^{\frac{1}{2}} u(t) \right|_H^2 }\\
      \vs
      \ds{\leq \gamma_{2(\a+\beta)}\left| u(t) \right|_H \left|\frac{\partial u}{\partial t}(t) \right|_H + \left|\psi(t) \right|_H \left| \frac{\partial u}{\partial t}(t) \right|_H.}
      \end{array}\end{equation}
Now, since
\[  \gamma_{2(\a+\beta)} \left| u(t) \right|_H \left|\frac{\partial u}{\partial t}(t) \right|_H + \left|\psi(t) \right|_H \left| \frac{\partial u}{\partial t}(t) \right|_H
    \leq \frac{1}{2} \left| \frac{\partial u}{\partial t}(t) \right|_H^2 + \gamma_{2(\a+\beta)}^2 \left|u(t) \right|_H^2 + \left|\psi(t) \right|_H^2,
  \]
\eqref{m32} implies
  \begin{equation}
  \label{m34}
    \left|\frac{\partial u}{\partial t}(t) \right|_H^2 +\mu \frac{d}{dt} \left| \frac{\partial u}{\partial t}(t) \right|_H^2 + \frac{d}{dt} \left|(-A)^{\frac{1}{2}} u(t) \right|_H^2 \leq 2\gamma_{2(\a+\beta)}^2\left| u(t) \right|_H^2 + 2 \left|\psi(t) \right|_H^2.
  \end{equation}
Therefore, integrating this expression with respect to $t \in\,(-\infty,\tau)$, we obtain
  \begin{equation} \label{u-dot-int-bound}
  \begin{array}{l}
    \ds{\int_{-\infty}^\tau \left| \frac{\partial u}{\partial t}(t) \right|_H^2 dt +
    |u(\tau)|_{H^1}^2 + \mu \left| \frac{\partial u}{\partial t}(t) \right|_H^2 }\\
    \vs
    \ds{\leq  2\gamma_{2(\a+\beta)}^2 \int_{-\infty}^\tau \left|u(t) \right|_H^2 dt + 2 \int_{-\infty}^\tau \left|\psi(t) \right|_H^2 dt,}
  \end{array}
  \end{equation}
  since, due to Lemma \ref{wave-high-reg-conv-to-zero-lem},
  \[\begin{array}{l}
  \ds{\int_{-\infty}^\tau \frac d{dt}\left( \mu \left| \frac{\partial u}{\partial t}(t) \right|_H^2 + \left| (-A)^{\frac{1}{2}} u(t) \right|_H^2 \right) dt }\\
  \vs
  \ds{= \mu \le|\frac{\partial u}{\partial t}(\tau)\r|_H^2+|u(\tau)|_{H^1}^2- \lim_{T \to -\infty} \left( \mu \left| \frac{\partial u}{\partial t}(T) \right|_H^2 + \left|  u(T) \right|_{H^1}^2 \right) = \mu \le|\frac{\partial u}{\partial t}(\tau)\r|_H^2+|u(\tau)|_{H^1}^2.}
   \end{array}
  \]
  Next we take the inner product of each side of \eqref{m33} with $u(t)$ and use the fact that
  \[
    \left< \frac{\partial^2 u}{\partial t^2}(t), u(t) \right>_H = \frac{1}{2} \frac{d^2}{dt^2} \left|u(t) \right|_H^2 - \left| \frac{\partial u}{\partial t}(t) \right|_H^2
  \]
  and again the Lipschitz-continuity of $B$ in $H^{2(\a+\beta)}$ to get
  \[\begin{array}{l}
  \ds{
    \frac{\mu}{2} \frac{d^2}{dt^2} \left|u(t) \right|_H^2 + \frac{1}{2}\frac{d}{dt} \left| u(t) \right|_H^2 + \hat{\gamma} \left|u(t) \right|_{H^1}^2
    \leq \mu \left| \frac{\partial u}{\partial t}(t) \right|_H^2 + \left< \psi(t), u(t) \right>_H}\\
    \vs
    \ds{\leq \mu \left| \frac{\partial u}{\partial t}(t) \right|_H^2 +\frac {\hat{\gamma}}{2}\,|u(t)|_{H^1}^2+ c\,|\psi(t)|^2_H,}
    \end{array}
  \]
  where $\hat{\gamma}:=1-\gamma_{2(\a+\beta)})/\a_1>0$.
This yields
  \begin{equation}
  \label{m40}
    \mu \frac{d^2}{dt^2} \left|u(t) \right|_H^2 + \frac{d}{dt} \left| u(t) \right|_H^2 + \hat{\gamma}\left| u(t) \right|_{H^1}^2
    \leq 2 \mu \left| \frac{\partial u}{\partial t}(t) \right|_H^2 + c\,\left| \psi(t) \right|_H^2.
  \end{equation}
Combining together \eqref{m34} and \eqref{m40}, we get
\[
\begin{array}{l}
\ds{\mu \frac{d^2}{dt^2} \left| u(t) \right|_H^2 + \frac{d}{dt} \left|u(t) \right|_H^2 + \hat{\gamma} \left|u(t) \right|_{H^1}^2+2\,\mu^2\frac d{dt}\le|\frac{\partial u}{\partial t}(t)\r|_H^2+2\mu\frac d{dt}\le|u(t)\r|_{H^1}^2}\\
\vs
\ds{\leq c_1\,\mu\,|u(t)|_H^2+c_2(1+\mu) \left| \psi(t) \right|_H^2.}
  \end{array}\]
If we take \[\mu<\frac{\hat{\gamma}\a_1}{2c_1},\]
and integrate both sides with respect to $t \in\,(-\infty,\tau)$, as a consequence of \eqref{m30}, we get
  \begin{equation}
  \label{m35}
\frac 12 \int_{-\infty}^\tau |u(t)|_{H^1}^2 dt \leq -2\mu \left< u(\tau), \frac{\partial u}{\partial t}(\tau) \right>_H +c_2(1+\mu) \int_{-\infty}^\tau \left|\psi(t) \right|_H^2 dt.
  \end{equation}
Substituting this  back into \eqref{u-dot-int-bound}, we have
  \[\begin{array}{l}
  \ds{
    \int_{-\infty}^\tau \left( \left| \frac{\partial u}{\partial t}(t) \right|_H^2 + \left|u(t) \right|_{H^1}^2 \right) dt + \mu \le|\frac{\partial u}{\partial t}(\tau)\r|_H^2+|u(\tau)|_{H^1}^2 dt}\\
    \vs
     \ds{\leq -c\,\mu \left< u(\tau),\frac{\partial u}{\partial t}(\tau) \right>_H+c\int_{-\infty}^\tau \left| \psi(t) \right|_H^2 dt.} \\
  \vs
  \ds{\leq c\,\sqrt{\mu} \left( \mu \le|\frac{\partial u}{\partial t}(\tau)\r|_H^2+|u(\tau)|_{H^1}^2 \right) + c \int_{-\infty}^\tau |\psi(t)|_H^2 dt}
    \end{array}
  \]
  Therefore,  since
  \[| \psi(t) |_H\leq c\,|\psi_1(t)-\psi_2(t)|_{H^{2\a}},\]
  and
  \[ \left|I_\mu (z^\mu_{\psi_1}(\tau) - z^\mu_{\psi_2}(\tau)) \right|_{\H_{1 + 2(\alpha + \beta)}} = \mu \le|\frac{\partial u}{\partial t}(\tau)\r|_H^2+|u(\tau)|_{H^1}^2  \]
if we choose $\mu_0$ small enough this yields \eqref{m20}.
\end{proof}
\begin{Remark}
{\em
\begin{enumerate}
\item
Notice that, since $B(0)=0$, we have $z^\mu_0=0$, so that from \eqref{m20} we get
\begin{equation}
\label{m43}
\begin{array}{l}
\ds{|z^\mu_\psi|^2_{L^2((-\infty,\tau);\H_{1+2(\a+\beta)})} + \sup_{t\leq \tau}  \left|I_\mu z^\mu_{\psi}(t) \right|_{\H_{1+ 2(\alpha +\beta)}}^2 \leq c\,|\psi|^2_{L^2((-\infty,\tau);H^{2\a})}},
\end{array}
\end{equation}
for any $\mu\leq \mu_0$ and $\tau\leq 0$.
\item By proceeding as in the proof of Lemma \ref{energy-est-lemma}, we can prove that
\begin{equation}
\label{m20bis}
\begin{array}{l}
\ds{|z^\mu_{\psi_1}-z^\mu_{\psi_2}|_{L^2((-\infty,\tau);\H_{2\beta})}^2 + \sup_{t \leq \tau} \left|I_\mu (z^\mu_{\psi_1}(t) - z^\mu_{\psi_2}(t)) \right|_{\H_{2\beta}}^2\leq c\,|\psi_1-\psi_2|_{L^2((-\infty,\tau);H^{-1})}^2.}
\end{array}
\end{equation}
and
\[|z^\mu_{\psi}|_{L^2((-\infty,\tau);\H_{2\beta})}^2 + \sup_{t \leq \tau} \left|I_\mu z^\mu_{\psi}(t)  \right|_{\H_{2\beta}}^2\leq c\,|\psi|_{L^2((-\infty,\tau);H^{-1})}^2.\]
\end{enumerate}
}
\end{Remark}

\section{A characterization of the quasi-potential}
For any $t_1<t_2$, $\mu>0$ and $z \in\,C((t_1,t_2);\H)$, we define
\begin{equation}
\label{m5}
I^{\mu}_{t_1,t_2}(z) =\frac{1}{2}  \inf \left\{ | \psi|_{L^2((t_1,t_2);H)}^2\,:\, z=z^\mu_{\psi,z_0}\right\},
\end{equation}
 where $z^\mu_{\psi,z_0}$ is a mild solution of the skeleton equation associated with equation  \eqref{abstract}, with deterministic control $\psi \in\,L^2((t_1,t_2);H)$ and initial conditions $z_0$, namely
\begin{equation}
\label{m4}
\frac{dz^\mu_{\psi,z_0}}{dt}(t)=A_\mu z^\mu_{\psi,z_0}(t)+B_\mu(z^\mu_{\psi,z_0}(t))+Q_\mu\psi(t),\ \ \ \ t_1\leq t\leq t_2.
\end{equation}
As in Definition \ref{def2.4}, for $\e,\mu>0$ and $z_0 \in\,\H$ we denote by $z^\mu_{\e,z_0} \in\,L^2(\Omega;C([0,T];\H))$ the mild solution of equation \eqref{abstract}.
Since the mapping $B_\mu:\mathcal{H}\to\mathcal{H}$ is Lipschitz-continuous and the noisy perturbation in \eqref{abstract} is of additive type, as an immediate consequence of  the contraction lemma, for any fixed $\mu>0$ the family $\{\mathcal{L}(z^\mu_{\e,z_0})\}_{\e>0}$  satisfies a large deviation principle in $C([t_1,t_2];\H)$, with action functional  $I^{\mu}_{t_1,t_2}$. In particular,
for any $\delta>0$ and $T>0$,
  \begin{equation} \label{LDP-lower-bound}
    \liminf_{\epsilon \to 0}  \epsilon \log \left( \inf_{z_0 \in \H} \P \left( \left|z^\mu_{\e,z_0}- z^{\mu}_{\psi,z_0} \right|_{C([0,T];\H)} < \delta \right) \right) \ge - \frac{1}{2} |\psi|_{L^2((0,T);H)}^2
  \end{equation}
  and, if $K^\mu_{0,T}(r) = \{z \in C([0,T];\H): I^\mu_{0,T}(z) \leq r \}$,
  \begin{equation} \label{LDP-upper-bound}
    \limsup_{\epsilon \to 0} \epsilon \log \left( \sup_{z_0 \in \H} \P \left( \textnormal{dist}_\H(z^{\mu}_{\epsilon,z_0}, K^\mu_{0,T}(r)) > \delta  \right) \right) \leq - r.
  \end{equation}

Analogously, if for any $\e>0$ $u_\e$ denotes the mild solution of equation \eqref{abstract-heat}, the family $\{\mathcal{L}(u_\e)\}_{\e>0}$  satisfies a large deviation principle in $C([t_1,t_2];H)$ with action functional
\begin{equation}
\label{m66}
I_{t_1,t_2}(\varphi) = \inf \left\{ \frac{1}{2} \left| \psi \right|_{L^2([t_1,t_2];H)}^2\, :\,
                               \varphi=\varphi_{\psi}\r\},\end{equation}
where     $\varphi_{\psi}$ is a mild solution of the skeleton equation associated with equation \eqref{abstract-heat}
\[\frac{du}{dt}(t)=A u(t)+B(u(t))+Q\psi(t),\ \ \ \ t_1\leq t\leq t_2.\]
In particular, the functionals $I^{\mu}_{t_1,t_2}$ and $I_{t_1,t_2}$ are lower semi-continuous and have compact level sets. Moreover, it is not difficult to show that for any compact sets $E\subset H$ and $\mathcal{E}\subset \mathcal {H}$, the level sets
\[ K_{E,t_1,t_2}(r) = \left\{\varphi \in C([t_1,t_2];H)\ ;\  I_{t_1,t_2}(\varphi) \leq r,\ \varphi(t_1) \in\, E \right\}
\]
and
\[
K^\mu_{\mathcal{E},t_1,t_2}(r) = \left\{ z \in C([t_1,t_2];\H)\ ;\  I^{\mu}_{t_1,t_2}(z) \leq r,\ z(t_1) \in \,\mathcal{E} \right\}
\]
are compact.

In what follows, for the sake of brevity, for any $\mu>0$  and $t \in\,(0,+\infty]$  we shall define $I^{\mu}_t:=I^{\mu}_{0,t}$ and $I^{\mu}_{-t}:=I^{\mu}_{-t,0}$ and, analogously, for any $t \in\,(0,+\infty]$ we shall define $I_t:=I_{0,t}$ and $I_{-t}:=I_{-t,0}$. In particular, we shall set
\[I^\mu_{-\infty}(z)=\sup_{t>0}I^\mu_{-t}(z),\ \ \ \ \ I_{-\infty}(\varphi)=\sup_{t>0}I_{-t}(\varphi).\]
Moreover, for any $r>0$ we shall set
\[K^\mu_{-\infty}(r)=  \left\{ z \in C((-\infty,0];\H) \ ;\  \lim_{t \to -\infty} |z(t)|_\H =0,\   I^\mu_{-\infty}(z) \leq r \right\}\]
and
\[  K_{-\infty}(r) = \left\{\varphi \in C((-\infty,0];H)\ ;\  \lim_{t \to -\infty} |\varphi(t)|_H =0,\  I_{-\infty}(\varphi) \leq r \right\}.\]

Once we have introduced the action functionals $I^{\mu}_{t_1,t_2}$ and $I_{t_1,t_2}$, we can introduce the corresponding {\em quasi-potentials}, by setting for any $\mu>0$ and $(x,y) \in\,\H$
\[V^\mu(x,y) = \inf \left\{ I^\mu_{0,T}(z) \ ;\  z(0) = 0,\  z(T) =(x,y), T > 0 \right\},\]
and for any $x \in\,H$
\[  V(x) = \inf \left\{ I_{0,T}(\varphi)\ ;\  \varphi(0)=0,\  \varphi(T) = x, \ T\ge 0 \right\}.\]
Moreover, for any $\mu>0$ and $x \in\,H$, we shall define
  \begin{equation} \label{V_mu-def}
    V_\mu(x) = \inf_{y \in H^{-1}} V^\mu(x,y).
  \end{equation}
  In \cite[Proposition 5.1]{cerrok} it has been proved that the level set $K_{-\infty}(r)$ is compact in the space $C((-\infty,0];H)$, endowed with the uniform convergence on bounded sets, and in \cite[Proposition 5.4]{cerrok} it has been proven that
  \[V(x)=\min \le\{\,I_{-\infty}(\varphi)\ ;\ \varphi \in\,C((-\infty,0];H),\ \lim_{t \to -\infty} |\varphi(t)|_H =0, \ \varphi(0)=x\,\r\}.\]
In what follows we want to prove an analogous result for $K^\mu_{-\infty}$, $V^\mu(x,y)$ and $V_\mu(x)$.

\begin{Theorem} \label{compact-level-sets-half-line-thm}
  For small enough $\mu>0$, the level sets $K^\mu_{-\infty}(r)$ are compact in the topology of uniform convergence on bounded intervals.
\end{Theorem}

\begin{proof}
Suppose that $z_n$ is a sequence in $K^\mu_{-\infty}(r)$ where $\mu\leq \mu_0$ and $\mu_0$ is the constant introduced in Lemma \ref{energy-est-lemma}. Let $c$  be the constant from that lemma
and let
  \begin{equation*}
    \mathcal{E}: = \left\{z \in \H: |C_\mu^{-\frac{1}{2}} z |_\H \leq  \sqrt{2cr}  \right\}
  \end{equation*}
  By Lemma \ref{energy-est-lemma},  $z_n \in K^\mu_{\mathcal{E},-N,0}(r)$, for any $N \in \mathbb{N}$.  Since $\mathcal{E}$ is compact in $\H$, in view of what we have seen above $K^\mu_{\mathcal{E},-N,0}(r)\subset C([-N,0];\H)$ is compact, for each $N \in\,\nat$. Then, by using a diagonalization procedure,  we can find a subsequence of $\{z_n\}$ that converges uniformly to a limit  $z^\mu \in\,C((-\infty,0];\H)$, uniformly on $[-N,0]$ for all $N$.  This means that there exist controls $\psi_N$ such that for $t \in\,[-N,0]$,
\[    z^\mu(t) =\Smu(t+N) z^\mu(-N)+\int_{-N}^t \Smu(t-s)  B_\mu(z^\mu(s)) ds +\int_{-N}^t \Smu(t-s) \Qmu \psi_N(s)  ds\]
and
\[ \frac{1}{2} |\psi_N|_{L^2([-N,0];H)}^2 \leq r.\]
  All of these $\psi_N$ coincide, because if $\varphi = \Pi_1 z^\mu$ satisfies the above equation,
  \[ \psi_N(t) = Q^{-1} \left( \mu \frac{\partial^2 \varphi}{\partial t^2}(t) + \frac{\partial \varphi}{\partial t}(t) - A \varphi(t) - B(\varphi(t))  \right) \]
  weakly.  Therefore, we can let $\psi=\psi_N$ and notice that
  \[\frac{1}{2}|\psi|_{L^2((-\infty,0);H)}^2 \leq r \]
This implies that
  for each $N_0\in\,\nat$
\[z^\mu(t) = \Smu(t+N_0) z^\mu(-N_0) + \int_{-N_0}^t \Smu(t-s) B_\mu(z^\mu(s)) ds+ \int_{-N_0}^t \Smu(t-s) \Qmu \psi(s)  ds.\]
Thus, by taking the limit as $N_0 \to +\infty$, we conclude that
\[    z^\mu(t) = \int_{-\infty}^t \Smu(t-s) B_\mu(z^\mu(s))  ds +  \int_{-\infty}^t \Smu(t-s) \Qmu \psi(s)  ds,\ \ \ \ t\leq 0.
  \]

Lastly, we need to show that
  \begin{equation*}
    \lim_{t \to -\infty} |z^\mu(t)|_\H= 0.
  \end{equation*}
  By \eqref{m43}, each $z_n$ has the property that
\[|z_n|_{L^2((-\infty,0);\H)} \leq c \sqrt{r}.
\]
Since $z_n \to z^\mu$ uniformly in $C((-N,0);\H)$ for each $N$,
\[|z^\mu|_{L^2((-\infty,0);\H)} = \lim_{N \to +\infty} |z^\mu|_{L^2((-N,0);\H)} \leq c \sqrt{r}.    \]

Next, by \eqref{control-term-goes-to-zero} and Hypothesis \ref{H2},
\[ \left|z^\mu(t) \right|_{\H_1} = \left| \int_{-\infty}^t \Smu(t-s) \left( B_\mu( z^\mu(s)) + Q\mu \psi(s) \right) ds \right|_{H_1}
 \leq c \left| z^\mu \right|_{L^2((-\infty,t);\H)} + c\left|\psi \right|_{L^2((-\infty,t);H^{-2\beta})}.\]
Because $z^\mu \in L^2((-\infty,0);\H)$, and $\psi \in L^2((-\infty,0);H)$,
\[ \lim_{t \to -\infty} \left| z^\mu(t) \right|_{\H_1} =0.
\]

\end{proof}

\begin{Corollary} \label{path-existence-corollary}
There exists $\mu_0>0$ such that  for any $\psi \in L^2((-\infty,0);H)$ and $\mu\leq \mu_0$ there exists $z^\mu_\psi \in C((-\infty,0];\H)$ such that
 \begin{equation}
 \label{m56}
   z^\mu_\psi(t)= \int_{-\infty}^t \Smu(t-s)  B_\mu(z^\mu_\psi(s))  ds+ \int_{-\infty}^t \Smu(t-s) \Qmu \psi(s)  ds,\ \ \ \ t\leq 0,
   \end{equation}
Moreover,
 \[
   \lim_{t \to -\infty} |z_\psi^\mu(t)|_\H =0.
 \]
\end{Corollary}

\begin{proof}
A standard fixed point argument shows that for any $\mu>0$ and $ N \in\,\nat$ there exists $z^{\mu}_{N} \in C([-N,0];\H)$ satisfying
   \[
     z^{\mu}_{N}(t) = \int_{-N}^t \Smu(t-s) B_\mu(z^{\mu}_{N}(s))  ds +\int_{-N}^t \Smu(t-s) \Qmu \psi(s)  ds.
   \]
Each $z^\mu_N$ can be seen as an element of $C((-\infty,0];H)$, just  by extending it to  $z^{\mu}_{N}(t) = 0$, for all $t< -N$.
   According to  Theorem \ref{compact-level-sets-half-line-thm}, there exists a subsequence $\{z^\mu_{N_k}\}$ converging to  some   $z^\mu \in K^\mu_{-\infty}\left(\frac{1}{2} |\psi|_{L^2((-\infty,0);H)}^2\right)$, uniformly on compact sets.  We notice that for any fixed $N_0 \in\,\nat$ and $t \ge -N_0$
   \[
     z^{\mu}_{N}(t) = \Smu(t+ N_0)z^\mu_N(-N_0)+\int_{-N_0}^t \Smu(t-s) B_\mu(z^{\mu}_{N}(s))  ds +\int_{-N_0}^t \Smu(t-s) \Qmu \psi(s)  ds.
   \]
Therefore, by taking the limit as $N \to +\infty$, we obtain
   \[
     z^{\mu}(t) = \Smu(t+ N_0)z^{\mu}(-N_0)+ \int_{-N_0}^t \Smu(t-s) B_\mu(z^{\mu}(s))  ds +\int_{-N_0}^t \Smu(t-s) \Qmu \psi(s)  ds.
   \]
Finally, if we let $N_0 \to +\infty$, we see that $z^\mu$ solves equation \eqref{m56}.

\end{proof}

As $K_{-\infty}(r)$ is compact in $C((-\infty,0];H)$ with respect to the uniform convergence on bounded intervals, we have analogously that
for any $\varphi \in\,L^2((-\infty,0)$ there exists $\varphi_\psi \in\,C((-\infty,0];H)$ such that
\[    \varphi_\psi(t) = \int_{-\infty}^t e^{(t-s)A} B(\varphi(s)) ds + \int_{-\infty}^t e^{(t-s)A} Q \psi(s) ds,
 \]
 and
 \[\lim_{t\to-\infty}|\varphi_\psi(t)|_H=0.\]

In \cite{cerrok}, it has been proved that the $V(x)$ can be characterized as
  \[
    V(x) = \inf \left\{ I_{-\infty}(\varphi) : \lim_{t \to -\infty} \varphi(t) = 0,  \varphi(0) = x \right\}.
  \]
  Here, we want to prove that an analogous result holds for $V^\mu(x,y)$ and $V_\mu(x)$, at least for $\mu$ sufficiently small.

\begin{Theorem} \label{quasipotential-representation-thm}
For small enough $\mu>0$, we have the following representation for the quasipotentials $V^\mu(x,y)$
  \begin{equation}
    \label{m69}
    V^\mu(x,y) = \min \left\{ I^\mu_{-\infty}(z): \lim_{t \to -\infty} |z(t)|_\H =0, z(0)=(x,y) \right\},
  \end{equation}
and for $V_\mu(x)$
  \begin{equation}
  \label{m68}
    V_\mu(x) = \min \left\{ I^\mu_{-\infty}(z): \lim_{t \to -\infty} |z(t)|_\H =0, \Pi_1z(0) = x \right\},
  \end{equation}
  whenever these quantities are finite.
\end{Theorem}

\begin{proof}
  From the definitions of $I^\mu_{t_1,t_2}$, it is clear that
 \[
    V^\mu(x,y) = \inf \left\{ I^\mu_{t_1,0}(z): z(t_1) = 0, z(0) =(x,y), t_1 \leq 0 \right\}.
  \]
Now, if we define
  \begin{equation}
    M^\mu(x,y) = \inf \left\{ I^\mu_{-\infty}(\varphi): \lim_{t \to -\infty} |z(t)|_\H =0, z(0) = (x,y)  \right\},
  \end{equation}
it is immediate to check that $M^\mu(x,y) \leq V^\mu(x,y)$, for any $(x,y) \in\,\H$.  To see this, we observe that if $z \in C([t_1,0];\H)$, with $z(t_1) = 0$ and $z(0)=(x,y)$, then
  \begin{equation}
    \hat{z} (t) = \begin{cases}
                            0, &t \leq t_1 \\
                            z(t), & t_1 < t \leq 0
                          \end{cases}
  \end{equation}
  has the property that $\hat{z}(0)=(x,y)$, and $|\hat{z}(t)|_{\H}\to 0$, as $t\to -\infty$. Moreover,
  \[I^\mu_{-\infty}(\hat{z}) = I^\mu_{t_1,0}(z).\]
Therefore, we need to show that $V^\mu(x,y) \leq M^\mu(x,y)$, for all $(x,y) \in\,\H$.

If $M^\mu(x,y) = +\infty$ there is nothing to prove. So, assume that $M^\mu(x,y) < +\infty$.  In view of Theorem \ref{compact-level-sets-half-line-thm},  there is a minimizer $z^\mu \in\,C((-\infty,0];\H_{1+2\beta})$, with $z^\mu(0)=(x,y)$ such that
  \[M^\mu(x,y) = I^\mu_{-\infty}(z^\mu).\]
  Moreover, thanks to \eqref{m30}
  \[\lim_{t \to -\infty} |z^\mu(t)|_{\H_{1+2\beta}} =0.\]
This means that for $\e>0$ fixed, there exists $t_\e<0$ such that
\[|z^\mu(t)|_{\H_{1+2\beta}}<\e,\ \ \ \ t\leq t_\e.\]
Now, let us denote $z_\e=z^\mu(t_\e)$ and let us define
\[    \psi_\epsilon = (L^\mu_{t_\epsilon - T_\mu, t_\epsilon})^{-1} z_\epsilon,
  \]
where $T_\mu>0$ is the time introduced in Theorem   \ref{L-inverse-thm}. Then, by Theorem \ref{L-inverse-thm}
  \begin{equation} \label{psi_epsilon-L2-bound}
| \psi_\epsilon |_{L^2((t_\epsilon -T_\mu, t_\epsilon);H)} \leq c(\mu,T_\mu) | z_\epsilon |_{\H_{1+2\beta}} \leq \epsilon\,c(\mu,T_\mu).
  \end{equation}

Next, for $t\in\,[t_\e-T_\mu,t_\e]$, we define
\[\zeta_\epsilon^\mu(t) = \int_{t_\epsilon - T_\mu}^t \Smu(t-s) \Qmu \psi_\epsilon(s)  ds.\]
Clearly we have $\zeta^\mu_\e(t_\e-T_\mu)=0$ and $\zeta^\mu_\e(t_\e)=z_\e.$
Moreover, thanks to \eqref{m17}, we have
\[|\zeta^\mu_\e(t)|_{\H_{1+2\beta}}\leq \frac{M_\mu}{\mu}\int_{t_\e-T_\mu}^{t}e^{-\omega_\mu(t-s)}|Q\psi_\e(s)|_{H^{2\beta}}\,ds\leq \frac{cM_\mu}\mu\int_{t_\e-T_\mu}^{t}e^{-\omega_\mu(t-s)}|\psi_\e(s)|_{H}\,ds,\]
so that, due to \eqref{psi_epsilon-L2-bound}
\begin{equation}
\label{m60}
\begin{array}{l}
\ds{ \int_{t_\epsilon - T_\mu}^{t_\epsilon} \left| \zeta^\mu_\epsilon(t)\right|_{\H_{1+2\beta}}^2 dt\leq \le(\frac{c\,M_\mu}\mu\r)^2 \int_{t_\epsilon - T_\mu}^{t_\epsilon} \left( \int_{t_\epsilon - T_\mu}^t M_\mu
 e^{-\omega_\mu (t-s)} \left|\psi_\epsilon(s) \right|_H ds \right)^2 dt }\\
 \vs
 \ds{
 \leq \le(\frac{M_\mu}{\omega_\mu \mu} \r)^2| \psi_\epsilon |_{L^2((t_\epsilon -T_\mu, t_\epsilon);H)}^2
 \leq \le(\frac{M_\mu}{\omega_\mu \mu} \r)^2 c(\mu,T_\mu)^2 \epsilon^2.}
  \end{array}\end{equation}
Since
\[\zeta^\mu_\e(t)= \int_{t_\epsilon - T_\mu}^{t} S_\mu(t-s)B_\mu(z^\mu_\e(s))\,ds+ \int_{t_\epsilon - T_\mu}^{t} S_\mu(t-s)Q_\mu\le(\psi_\e(s)-Q^{-1}B(\Pi_1 z^\mu_\e(s))\r)\,ds,\]
we have
\[I^\mu_{t_\e-T_\mu,t_\e}(\zeta^\mu_\e)\leq 2\,|\psi_\e|^2_{L^2((t_\e-T_\mu,t_\e);H)}+2\,|Q^{-1}B(\Pi_1 z^\mu_\e)|^2_{L^2((t_\e-T_\mu,t_\e);H)}.\]
Then, as due to Hypothesis \ref{H2}
\[|Q^{-1}B(\Pi_1 \zeta^\mu_\e(s))|_H\leq c\,|B(\Pi_1 \zeta^\mu_\e(s))|_{H^{2\beta}}\leq c\,\gamma_{2\beta}\,|\Pi_1 \zeta^\mu_\e(s)|_{H^{2\beta}}\leq c\,\gamma_{2\beta}\,|\zeta^\mu_\e(s)|_{\H_{2\beta}},\]
thanks to \eqref{psi_epsilon-L2-bound} and \eqref{m60}, we can conclude
\begin{equation}
\label{m70}
I^\mu_{t_\e-T_\mu,t_\e}(\zeta^\mu_\e)\leq c_\mu\,\e^2.\end{equation}
Finally,  we define
  \begin{equation}
    \hat{\zeta}^\mu_\epsilon(t) = \begin{cases}
                                    \zeta^\mu_\epsilon(t), & t_\epsilon -T_\mu \leq t \leq t_\epsilon \\
                                    z^\mu(t), & t>t_\epsilon.
                                  \end{cases}
  \end{equation}
 It is immediate to check that $\hat{\zeta}^\mu_\e \in\,C([t_\e-T_\mu,0];\H)$, $\hat{\zeta}^\mu_\e=0$ and $\hat{\zeta}^\mu_\e(0)=(x,y)$. Moreover, thanks to \eqref{m70}
  \begin{equation}
    I^\mu_{t_\epsilon -T_\mu, 0}( \hat{\zeta}^\mu_\epsilon) \leq I^\mu_{-\infty}(z^\mu) + I^\mu_{t_\epsilon - T_\mu, t_\epsilon}(\zeta^\mu_\epsilon)
    =M^\mu(x,y) + I^\mu_{t_\epsilon -T, t_\epsilon}(\zeta^\mu_\epsilon)\leq M^\mu(x,y)+c_\mu\,\e^2.
  \end{equation}
Due to the arbitrariness of $\e>0$, this implies
\[    V^\mu(x,y) \leq M^\mu(x,y),
  \]
and then \eqref{m69} follows.

Finally, in order to prove \eqref{m68}, we just notice that there exists $\{y_n\}\subset H^{-1}$  such that
\[V_\mu(x)=\lim_{n\to\infty}V^\mu(x,y_n)\]
and
\[V^\mu(x,y_n)=I_{-\infty}^\mu(z_n),\]
for some  $\{z_n\}\subset C((-\infty,0];\H)$ such that $z_n(0)=(x,y_n)$ and
\[\lim_{t\to-\infty}|z_n(t)|_\H=0.\]
As
\[\sup_{n \in\,\nat}I^\mu_{-\infty}(z_n)<\infty,\]
due to Theorem \ref{compact-level-sets-half-line-thm} we have that
there exists a subsequence $\{z_{n_k}\}$ which is uniformly convergent on bounded sets to some $z \in\,C((-\infty,0];\H)$. In particular, $\Pi_1 z(0)=x$ and $|z(t)|_\H\to 0$, as
$t\to-\infty$. Since $I^\mu_{-\infty}$ is lower semi-continuous, we have
\[I^\mu_{-\infty}(z)\leq \liminf_{k\to\infty} I^\mu_{-\infty}(z_{n_k})=V_\mu(x),\]
and then $V_\mu(x)=I^\mu_{-\infty}(z)$, so that \eqref{m68} holds true.
\end{proof}

The characterization of $V^\mu(x,y)$ and $V_\mu(x)$ given in Theorem \ref{quasipotential-representation-thm}, implies that $V^\mu$ and $V_\mu$ have compact level sets.

\begin{Theorem}
\label{fine1}
For any $\mu>0$ and $r\geq 0$ the level sets
\[K^\mu(r)=\le\{ (x,y) \in\,\H\,:\ V^\mu(x,y)\leq r\r\}\]
and
\[K_\mu(r)=\le\{ x \in\,H\,:\ V_\mu(x)\leq r \r\}\]
are compact, in $\H$ and $H$, respectively.
\end{Theorem}

\begin{proof}
We prove this result for $V^\mu$ and $K^\mu$, as the proof for $V_\mu$ and $K_\mu$ is completely analogous.
Let $\{(x_n,y_n)\}_{n \in\,\nat}\subset K^\mu(r).$ In view of Theorem \ref{quasipotential-representation-thm}, for each $n \in\,\nat$ there exists $z^n \in\,C((-\infty,0];\H)$, with $z^n(0)=(x_n,y_n)$, and $|z^n(t)|_H\to 0$, as $t\downarrow -\infty$, such that $V^\mu(x_n,y_n)=I^\mu_{-\infty}(z^n).$ As $I^\mu_{-\infty}(z^n)\leq r$ and the level sets of $I^\mu_{-\infty}$ are compact in $C((-\infty,0];\H)$, as shown in Theorem \ref{compact-level-sets-half-line-thm}, there exists a subsequence $\{z^{n_k}\}\subseteq \{z^n\}$ converging to some $\hat{z} \in\,
C((-\infty,0];\H)$, with $I_{-\infty}^\mu(\hat{z})\leq r$. Since
\[\lim_{k\to\infty}(x_{n_k},y_{n_k})=\lim_{k\to\infty}z^{n_k}(0)=\hat{z}(0)=:(\hat{x},\hat{y}),\ \ \ \text{in}\,\H,\]
due to Theorem \ref{quasipotential-representation-thm} we have
\[V^\mu(\hat{x},\hat{y})\leq I^\mu_{-\infty}(\hat{z})\leq r,\]
so that $(\hat{x},\hat{y}) \in\,K^\mu(r)$.

\end{proof}

\section{Continuity of $V^\mu$ and  $V_\mu$}
As a consequence of Theorem \ref{fine1},  the mappings $V^\mu:\H\to[0,+\infty]$ and  $V_\mu:H\to [0,+\infty]$ are lower semicontinuous.
 Our purpose here is to prove that the mappings
\[V^\mu:\H_{1+2\beta}\to [0,+\infty),\ \ \ V_\mu:H^{1+2\beta}\to[0,+\infty)\]
are well defined and  continuous, uniformly in $0<\mu<1$.

\begin{Lemma}
  Let  us fix $(x,y) \in\,\H_{1+2\beta}$ and   $\mu>0$ and let
$z(t) = \Smu(-t)(x,-y),$ $t \leq 0.$
  Then, if we denote $\varphi(t)=\Pi_1z(t),$ we have that   $\varphi$ is a weak solution to
    \begin{equation} \label{Pi1-Smu-weak-sol-eq}
 \le\{     \begin{array}{l}
\ds{        \mu \frac{\partial^2\varphi}{\partial t^2}(t) = A \varphi(t) + \frac{\partial \varphi}{\partial t}(t),\ \ \ \  t\leq 0}\\
\vs
\ds{\varphi(0)  =x, \quad \frac{\partial \varphi}{\partial t}(0) = y.}
      \end{array}\r.
    \end{equation}
and
\begin{equation}
\label{Pi1-Smu-control-eq}
\frac{1}{2} \int_{-\infty}^0 \left| Q^{-1} \left( \mu \frac{\partial^2\varphi}{\partial t^2}(t) + \frac{\partial\varphi}{\partial t}(t) - A \varphi(t) \right) \right|_H^2 dt = \left|(-A)^{\frac{1}{2}}Q^{-1} x \right|_H^2 + \mu \left| Q^{-1} y \right|_H^2.
    \end{equation}
Moreover, $\varphi \in L^2((-\infty,0);H^{1+2\beta})$ and
    \begin{equation}  \label{Pi1-Smu-L2-eq}
      \int_{-\infty}^0 \left|\varphi(t) \right|_{H^{1+2\beta}}^2 dt \leq c\,(1+\mu+\mu^2)|(x,y)|_{\H_{1+2\beta}}^2.
    \end{equation}
\end{Lemma}

\begin{proof}
  The weak formulation \eqref{Pi1-Smu-weak-sol-eq} is clear because for $t<0$
  \[
\frac{\partial z}{ \partial t} (t)= - \Amu \Smu(-t) (x,-y)= (-\Pi_2 z(t), - \frac{1}{\mu} A \varphi(t) + \frac{1}{\mu} \Pi_2 z(t)),
  \]
so that
\[    \mu \frac{\partial^2 \varphi}{\partial t^2} (t) = A \varphi(t) + \frac{\partial\varphi}{\partial t}(t).
  \end{equation*}
Moreover,
\[    \frac{\partial \varphi}{\partial t} (0) = - \Pi_2 z(0) = y.
  \end{equation*}

Now, property \eqref{Pi1-Smu-control-eq} can be proven by noticing that
  \[\begin{array}{l}
  \ds{\frac{1}{2} \int_{-\infty}^0 \left| Q^{-1} \left( \mu \frac{\partial^2\varphi}{\partial t^2}(t) + \frac{\partial\varphi}{\partial t}(t) - A \varphi(t) \right) \right|_H^2 dt}\\
  \vs
  \ds{= \frac{1}{2} \int_{-\infty}^0 \left|Q^{-1} \left( \mu \frac{\partial^2\varphi}{\partial t^2}(t) - \frac{\partial\varphi}{\partial t}(t) - A \varphi(t) \right) \right|_H^2 dt}\\
  \vs
  \ds{
  + 2 \int_{-\infty}^0 \left<Q^{-1} \frac{\partial\varphi}{\partial t}(t), Q^{-1}\left( \mu \frac{\partial^2\varphi}{\partial t^2}(t) -A \varphi(t) \right)  \right>_H  dt}\\
  \vs
\ds{= \left|Q^{-1}(-A)^{\frac{1}{2}} x \right|_H^2 + \mu \left|Q^{-1} y \right|_H^2-\lim_{t\to-\infty}|C_\mu^{-1/2}z(t)|_H^2.}
  \end{array}\]
  Then, \eqref{Pi1-Smu-control-eq} follows from \eqref{m17}, as
  \[|C_\mu^{-1/2}z(t)|_H\leq |z(t)|_{\H_{1+2\beta}}\leq M_\mu\,e^{-\omega_\mu t}|(x,y)|_{\H_{1+2\beta}}\to 0,\ \ \ \text{as}\ t\downarrow -\infty.\]

Finally, to obtain estimate \eqref{Pi1-Smu-L2-eq}, we notice that if
  \begin{equation*}
\varphi(t)=    \Pi_1 \Smu(-t) (x,-y)
\end{equation*}
then by \eqref{linear-eq-energy-method},
\[ \left| \varphi(t) \right|_{H^{1+2\beta}}^2 =  \frac{1}{2} \frac{d}{dt} \left|\varphi(t) - \mu \frac{\partial \varphi}{\partial t}(t)  \right|_{H^{2\beta}}^2 + \frac{\mu}{2}\frac{d}{dt} \left| \varphi(t) \right|_{H^{1 + 2 \beta}}^2.\]
Integrating, we obtain
\[ \int_{-\infty}^0 \left| \varphi(t) \right|_{H^{1 + 2\beta}}^2 dt = \frac{1}{2} \left| x + \mu y\right|_{H^{2\beta}}^2 + \frac{\mu}{2} \left|x \right|_{H^{1 + 2\beta}}^2,\]
which yields \eqref{Pi1-Smu-L2-eq}.

\end{proof}

As a consequence of the previous lemma, we obtain the following bounds for $V^\mu(x,y)$ and $V_\mu(x)$.
\begin{Corollary}
\label{fine13}
  For any $\mu>0$ and $(x,y) \in\,\H_{1+2\beta},$ we have
  \begin{equation}
  \label{m73}
    V^\mu(x,y) \leq  c(1+\mu+\mu^2) |(x,y)|_{\H_{1+2\beta}} 
    \end{equation}
  and
  \begin{equation}
  \label{m74}
    V_\mu(x) \leq c(1+\mu) \left|x \right|_{H^{1+2\beta}}^2
  \end{equation}
\end{Corollary}

\begin{proof}
  The proof is based on the fact that
  \begin{equation*}
    V^\mu(x,y) \leq I^\mu_{-\infty} ( \Pi_1 \Smu(- \cdot) (x, -y))
  \end{equation*}
  and
  \begin{equation*}
    V_\mu(x) \leq I^\mu_{-\infty} ( \Pi_1 \Smu(- \cdot)(x,0)).
  \end{equation*}
Now,  if we set $z(t) =\Smu(-t)(x,-y)$ and $\varphi(t)=\Pi_1z(t)$, due to Hypothesis \ref{H2} we have
  \[\begin{array}{l}
\ds{I^\mu_{-\infty}(z) = \frac{1}{2} \int_{-\infty}^0 \left|Q^{-1} \left( \mu \frac{\partial^2\varphi}{\partial t^2}(t) + \frac{\partial\varphi}{\partial t}(t) -A \varphi(t) - B(\varphi(t)) \right) \right|_H^2 dt }\\
\vs
\ds{\leq \int_{-\infty}^0 \left| Q^{-1} \left( \frac{\partial^2\varphi}{\partial t^2}(t) + \frac{\partial\varphi}{\partial t}(t) -A \varphi(t) \right) \right|_H^2 dt + c\,\gamma_{2\beta}^2 \int_{-\infty}^0 \left|\varphi(t) \right|_{H^{2\beta}}^2 dt.}
\end{array}
  \]
From \eqref{Pi1-Smu-control-eq} and \eqref{Pi1-Smu-L2-eq}, this give \eqref{m73}. Finally, \eqref{m74} is a consequence of \eqref{m73} and of the way $V_\mu(x)$ has been defined.
\end{proof}

Now, we can prove the continuity of  $V^\mu$ and $V_\mu$.

\begin{Theorem}\label{continuity-of-V-mu-tilde-thm}
For each $\mu>0$ the mappings $V^\mu:\H_{1+2\beta}\to [0,+\infty)$ and $V_\mu:H^{1+2\beta}\to [0,+\infty)$ are well defined and continuous. Moreover,
  \begin{equation} \label{fine12}
    \lim_{n \to \infty} \left|(x,y) - (x_n,y_n)\right|_{\H_{1+2\beta}}=0\Longrightarrow    \lim_{n \to \infty} \sup_{0 < \mu < 1} \left| V^\mu(x,y) - V^\mu(x_n,y_n) \right|   =0.
  \end{equation}
and
  \begin{equation} \label{continuity-of-V-mu-tilde-eq}
\lim_{n \to \infty} \left|x - x_n\right|_{H^{1+2\beta}}=0\Longrightarrow   \lim_{n \to \infty} \sup_{0 < \mu < 1} \left| V_\mu(x) - V_\mu(x_n) \right|   =0.
  \end{equation}
\end{Theorem}

\begin{proof}
In view of Corollary \ref{fine13}, if $(x,y) \in\,\H_{1+2\beta}$, then $V^\mu(x,y)<+\infty$ and  if $x \in\,H^{1+2\beta}$, then $V_\mu(x)<+\infty$. On the other hand, if $V^\mu(x,y)<+\infty$,  thanks to Theorem \ref{quasipotential-representation-thm} there exists $z^\mu \in\,C((-\infty,0];\H)$ such that
\[V^\mu(x,y) = I^\mu_{-\infty}(z^{\mu}),\ \ \ \ z^{\mu}(0) =(x,y).\]
According to Lemma \ref{wave-high-reg-conv-to-zero-lem}, this implies that $z^\mu \in\,C((-\infty,0];\H_{1+2\beta})$, so that $(x,y)=z^\mu(0) \in\,\H_{1+2\beta}$. Analogously, if $V_\mu(x)<+\infty$, we can prove that $x \in\,H^{1+2\beta}$, so that we can conclude that the mappings $V^\mu$ and $V_\mu$ are well defined in $\H_{1+2\beta}$ and $H^{1+2\beta}$, respectively.

Now, in order to prove \eqref{fine12}, by using again  Theorem \ref{quasipotential-representation-thm}, for each $n \in\,\nat$ we can find $z^\mu_n \in\,C((-\infty,0];\H)$ such that
\[V^\mu(x_n,y_n) = I^\mu_{-\infty}(z^{\mu}_{n}),\ \ \ \  z^{\mu}_{n}(0) =(x_n,y_n).\]
Then, if we define
\[\hat{z}^\mu_n(t)= \Smu(-t) (x - x_n,y-y_n),\]
and
\[{\varphi}^\mu_n(t)=\Pi_1 z^\mu_n(t),\ \ \ \ \hat{\varphi}^\mu_n(t)=\Pi_1\hat{z}^\mu_n(t),\ \ \ \ t\leq 0,\]
we have
$\hat{z}^\mu_n(0)=(x-x_n,y-y_n)$ and for any $\e>0$
\[\begin{array}{l}
\ds{V^\mu(x,y) \leq I^\mu_{-\infty}( z^{\mu}_{n} + \hat{z}^{\mu}_n) }\\
\vs
\ds{\leq \frac{1}{2} \int_{-\infty}^0 \le|Q^{-1} \left( \mu \frac{\partial^2\varphi^\mu_n}{\partial t^2}(t) + \frac{\partial\varphi^\mu_n}{\partial t}(t) -A \varphi^\mu_n(t) + B(\Pi_1\varphi^\mu_n(t)) \right) \r.}\\
\vs
\ds{+Q^{-1} \left(\mu \frac{\partial^2\hat{\varphi}^\mu_n}{\partial t^2}(t) + \frac{\partial\hat{\varphi}^\mu_n}{\partial t}(t) -A \hat{\varphi}^\mu_n(t) \right)+Q^{-1} \left(B({\varphi}^\mu_n+ \hat{\varphi}^\mu_n(t)) - B({\varphi}^\mu_n(t)) \right)  \Bigg|_H^2 dt }\\
\vs
\ds{
\leq (1 + \epsilon) I^\mu_{-\infty}(z^\mu_n) + (1 + \frac{1}{\epsilon}) \int_{-\infty}^0 \left| Q^{-1} \left(\frac{\partial^2\hat{\varphi}^\mu_n}{\partial t^2}(t)
     + \frac{\partial\hat{\varphi}^\mu_n}{\partial t}(t) - A \hat{\varphi}^\mu_n(t) \right) \right|_H^2 dt} \\
     \vs
     \ds{
+ c\,(1 + \frac{1}{\epsilon})\int_{-\infty}^0 \left|\hat{\varphi}^\mu_n(t) \right|_{H^{2\beta}}^2 dt.}
\end{array}\]
  Now, by \eqref{Pi1-Smu-control-eq} and \eqref{Pi1-Smu-L2-eq}, we see that for $0<\mu<1$
  \[\begin{array}{l}
  \ds{
    V_\mu(x,y) \leq ( 1 + \epsilon) V_\mu(x_n,y_n) }\\
    \vs
    \ds{+ c\,(1 + \frac{1}{\epsilon}) \left|(x - x_n,y-y_n) \right|_{\H_{1+2\beta}}^2
    +c^2(1 + \frac{1}{\epsilon})\left|(x-x_n,y-y_n) \right|_{\H_{2\beta}}.}
    \end{array}
  \]

  If we follow the same procedure with $z^\mu$ as the minimizer of $V^\mu(x,y)$ and
  \[\hat{z}^{\mu}_{n}(t) = \Smu(-t) (x_n -x,y-y_n),\]
  we see that for $0< \mu < 1$
  \[\begin{array}{l}
  \ds{
    V^\mu(x_n,y_n) \leq ( 1 + \epsilon) V^\mu(x,y) }\\
    \vs
    \ds{+ c\,(1 + \epsilon^{-1}) \left|(x - x_n,y-y_n) \right|_{\H_{1+2\beta}}^2+c\,(1 + \epsilon^{-1}) \left|(x-x_n,y-y_n) \right|_{\H_{2\beta}}.}
    \end{array}
  \]
  From these two estimates and Corollary \ref{fine13}, we see that
  \[\sup_{0< \mu <1} \left| V^\mu(x,y) - V^\mu(x_n,y_n) \right|  \leq c \epsilon \left|(x,y) \right|_{\H_{1+2\beta}}^2+c\,\le(1+\e^{-1}\r)|(x-x_n,y-y_n)|^2_{\H_{1+2\beta}},\]
  so that
 \[\limsup_{n\to\infty}  \sup_{0< \mu <1} \left| V^\mu(x,y) - V_\mu(x_n,y_n) \right| \leq c\, \epsilon \left|(x,y) \right|_{\H_{1+2\beta}}^2.\]
 Due to the arbitrariness  of $\epsilon>0$, \eqref{fine12} follows. The proof of \eqref{continuity-of-V-mu-tilde-eq} is completely analogous to the proof of \eqref{fine12} and for this reason we omit it.
\end{proof}

\section{Upper bound}
\label{sec6}

In this section we show that for any closed set $N \subset H$
\begin{equation}
  \limsup_{\mu \downarrow 0} \inf_{x \in N} V_\mu(x) \leq \inf_{x \in N} V(x).
\end{equation}
First of all, we we notice that if $I_{-\infty}(\varphi)<\infty$, then
\begin{equation}
\label{m48}
\varphi \in\,L^2((-\infty,0);H^{2(1+\beta)}),\ \ \ \frac{\partial \varphi}{\partial t} \in\,L^2((-\infty,0);H^{2\beta}),
\end{equation}
and
  \begin{equation}
  \label{m49}
    I_{-\infty}(\varphi) =\frac{1}{2} \int_{-\infty}^0 \left|Q^{-1} \left( \frac{\partial\varphi}{\partial t}(t) - A \varphi(t) -
    B(\varphi(t)) \right) \right|_H^2 dt.
  \end{equation}
Actually, if $\varphi$ solves
\begin{equation*}
    \varphi(t) = \int_{-\infty}^t e^{(t-s)A} B(\varphi(s)) ds + \int_{-\infty}^t e^{(t-s)A} Q \psi(s) ds
  \end{equation*}
  then we can check that \eqref{m48} holds and
    \begin{equation*}
    \psi(t) = Q^{-1} \left( \frac{\partial}{\partial t} \varphi(t) - A \varphi(t) - B(\varphi(t)) \right),
  \end{equation*}
so that \eqref{m49} follows.
Moreover, if
\[  \varphi \in\,L^2((-\infty,0);H^{2(1+\beta)}),\ \ \ \ \ \ \ \frac{\partial \varphi}{\partial t}, \frac{\partial^2 \varphi}{\partial t^2}\in\,L^2((-\infty,0);H^{2\beta}),\]
then
\[
    I^\mu_{-\infty} (z) = \frac{1}{2} \int_{-\infty}^0 \left|Q^{-1} \left( \mu \frac{\partial^2\varphi}{\partial t^2}{\varphi}(t) + \frac{\partial}{\partial t}{\varphi}(t) - A \varphi(t) - B(\varphi(t)) \right) \right|_H^2 dt,
  \]
  where
  \[z(t)=(\varphi(t),\frac{\partial \varphi}{\partial t}(t)).\]
Actually, if $I^\mu_{-\infty}(z)<\infty$, then $z$ solves
  \begin{equation*}
    z(t) = \int_{-\infty}^t \Smu(t-s) B_\mu(z(s))  ds
    +\int_{-\infty}^t \Smu(t-s) \Qmu \psi(s)ds
  \end{equation*}
  so that
  \begin{equation*}
    \psi(t) = Q^{-1} \left(\mu \frac{\partial^2\varphi}{\partial t^2} (t) + \frac{\partial\varphi }{\partial t} (t) - A \varphi(t)  - B(\varphi(t)) \right)
  \end{equation*}
  weakly.

In particular, as in \cite{cf}, where the finite dimensional case is studied, this means
\begin{equation} \label{I-mu=I+remainder}
\begin{array}{l}
\ds{  I^\mu_{-\infty} (z) = I_{-\infty}(\varphi) + \frac{\mu^2}{2} \int_{-\infty}^0 \left| Q^{-1} \frac{\partial^2\varphi}{\partial t^2}(t) \right|_H^2 dt}\\
\vs
\ds{  + \mu \int_{-\infty}^0 \left< Q^{-1} \frac{\partial^2\varphi}{\partial t^2}(t), Q^{-1} \left( \frac{\partial\varphi}{\partial t}(t) - A \varphi(t) - B(\varphi(t)) \right) \right>_H dt,}
\end{array}
\end{equation}
where $\varphi(t)=\Pi_1 z(t)$,
as long as all of these terms are finite.

Now, for any $\mu>0$ let us define
\begin{equation}
\label{m83}
  \rho_\mu(t) = \frac{1}{\mu^\alpha} \,\rho \left(\frac{t}{\mu^{\alpha}} \right),\ \ \ \ t \in\,\mathbb{R},
\end{equation}
for some $\alpha>0$ to be chosen later, where
 $\rho \in C^\infty(\mathbb{R})$ is the usual mollifier  function such that
\[\text{supp}(\rho) \subset \subset [0,2] ,\ \ \ \int_\mathbb{R} \rho(s) ds = 1,\ \ \ 0 \leq \rho \leq 1.\]
This scaling ensures that
\[
  \int_\mathbb{R} \rho_\mu(s) ds = 1.
\]
Next, we define $\varphi_\mu$ as the convolution
\begin{equation} \label{convolution-def}
  \varphi_\mu(t) = \int_{-\infty}^0 \rho_\mu(t-s) \varphi(s) ds.
\end{equation}

\begin{Lemma} \label{x_mu-conv-to-x-lem}
Assume that
\[\varphi \in L^2((-\infty,0);H^{2(1+\beta)}) \cap C((-\infty,0];H^{1+2\beta}),\ \ \   \frac{\partial\varphi}{\partial t} \in L^2((-\infty,0);H^{2\beta})\]
with
\[\varphi(0) =x \in H^{1+2\beta},\ \ \
\lim_{t\to -\infty}\left| \varphi(t) \right|_{H^{1+2\beta}}= 0.\]
Then,
\begin{equation}
\label{m80}
\varphi_\mu \in L^2((-\infty,0);H^{2(1+\beta)}) \cap C((-\infty,0];H^{1+2\beta}),\ \ \   \frac{\partial\varphi_\mu}{\partial t} \in L^2((-\infty,0);H^{2\beta}),
\end{equation}
and
\begin{equation}
\label{m82bis}   \lim_{t \to -\infty} \sup_{\mu >0} \left|\varphi_\mu(t) \right|_{H^{1+2\beta}} =0.
  \end{equation}
Moreover,
\[\frac{\partial^2\varphi_\mu}{\partial t^2} \in L^2((-\infty,0);H^{2\beta})\]
 and  for all $\mu >0$,
  \begin{equation}
  \label{m81}
    \left|\frac{\partial^2\varphi_\mu}{\partial t^2} \right|_{L^2((-\infty,0);H^{2\beta})} \leq \frac{c}{\mu^\alpha} \left|\frac{\partial\varphi_\mu}{\partial t}\right|_{L^2((-\infty,0);H^{2\beta})}.
  \end{equation}
  \end{Lemma}

\begin{proof}
Since we have
\[\varphi_\mu(t) = \int_{t - 2 \mu^\alpha}^t \rho_\mu(t-s)\varphi(s) ds,\]
it follows
\[    \int_{-\infty}^0 \left| \int_{t - 2 \mu^\alpha}^t \rho_\mu(t-s) \varphi(s) ds \right|_{H^{2(1+\beta)}}^2 dt \leq \int_{-\infty}^0 \left( \int_{0}^{2 \mu^\alpha} \rho_\mu^2(s) ds \right) \left( \int_{t - 2 \mu^\alpha}^t \left| \varphi(s) \right|_{H^{2(1+\beta)}}^2 ds \right) dt.
  \]
Therefore, as  \[
    \int_{0}^{2 \mu^\alpha} \rho^2_\mu(s) ds \leq \frac{2}{\mu^\alpha},
  \]
we get
  \begin{equation}
  \label{l1}
\begin{array}{l}
\ds{  \left|\varphi_\mu \right|_{L^2((-\infty,0);H^{2(1+\beta)})}^2 dt
    \leq \frac{2}{\mu^\alpha} \int_{-\infty}^0 \int_{t - 2 \mu^\alpha}^t  \left|\varphi(s) \right|_{H^{2(1+\beta)}}^2 ds dt
   }\\
    \vs
    \ds{ \leq \frac{2 \mu^\alpha}{\mu^\alpha} \int_{-\infty}^0 \left| \varphi(s) \right|_{H^{2(1+\beta)}}^2 ds=2\,|\varphi|^2_{L^2((-\infty,0);H^{2(1+\beta)})}.}
    \end{array}
  \end{equation}
Next,   since
\[\lim_{t \to -\infty} \left|\varphi(t) \right|_{H^{1+2\beta}}=0,\]
we have that  $\varphi:(-\infty,0] \to H^{1+2\beta}$ is uniformly continuous.  Therefore, as
  \[\left|\int_{-\infty}^{t_1} \rho_\mu(t_1-s) \varphi(s) ds - \int_{-\infty}^{t_2} \rho_\mu(t_2 -s)  \varphi(s) ds \right|_{H^{1+2\beta}}= \left|\int_0^\infty \rho_\mu(s) \left( \varphi(t_1-s) - \varphi(t_2 - s) \right)\,ds \right|_{H^{1+2\beta}},
  \]
we can conclude  that $\varphi_\mu$ is uniformly continuous too, with values in $H^{1+2\beta}$.
Finally, since
 \[   \frac{\partial {\varphi}_\mu}{\partial t}(t) = \int_0^\infty \rho_\mu(s) \frac{\partial{\varphi}}{\partial t}(t-s) ds,\]
  by proceeding as above we get
  \[\frac{\partial \varphi_\mu}{\partial t} \in\,L^2((-\infty,0);H^{2\beta}),\]
  so that, thanks to \eqref{l1}, we can conclude that \eqref{m80} holds true.

  Concerning \eqref{m82bis},   let us fix $\epsilon>0$.  Then there exists $T_\epsilon>0$ such that
  \[
    \left|\varphi(t) \right|_{H^{1+2\beta}} < \epsilon, \quad t< -T_\epsilon.
  \]
Then, for $t< -T_\epsilon$, we have
  \[\left| \int_{-\infty}^t \rho_\mu(t-s) \varphi(s) ds \right|_{H^{1+2\beta}}
    \leq \int_{-\infty}^t \rho_\mu(t-s) \left|\varphi(s) \right|_{H^{1+2\beta}} ds \leq \epsilon \int_{0}^\infty \rho_\mu(s) ds = \epsilon.\]
and this yields \eqref{m82bis}.

Finally, let us prove  \eqref{m81}.
As
\[    \frac{\partial{\varphi}_\mu}{\partial t}(t) = \int_{-\infty}^0 \rho_\mu(t-s) \frac{\partial {\varphi}}{\partial s}(s) ds,\]
we have
\[\frac{\partial^2{\varphi}_\mu}{\partial t^2}(t) = \int_{-\infty}^0 \frac{d}{dt}{\rho}_\mu(t-s) \frac{\partial{\varphi}}{\partial s}(s) ds  =\frac{1}{\mu^{2 \alpha}} \int_{-\infty}^0 \rho^\prime\left( \frac{t-s}{\mu^{\alpha}} \right) \frac{\partial\varphi}{\partial s}(s) ds.\]
This yields
  \[
  \begin{array}{l}
  \ds{\int_{-\infty}^0 \left| \frac{\partial^2 {\varphi}_\mu}{\partial t^2}(t) \right|_{H^{2\beta}}^2 dt = \frac{1}{\mu^{4 \alpha}} \int_{-\infty}^0 \left| \int_{t - 2 \mu^\alpha}^t \rho^\prime\left( \frac{t-s}{\mu^{\alpha}} \right) \frac{\partial}{\partial s}{\varphi}(s) ds \right|_{H^{2\beta}}^2 dt }\\
  \vs
  \ds{\leq  \frac{1}{\mu^{4 \alpha}} \int_{-\infty}^0 \left( \int_{t - 2 \mu^\alpha}^t \left( \rho^\prime \left(\frac{t-s}{\mu^\alpha} \right) \right)^2  ds \right) \left( \int_{t - 2 \mu^\alpha}^t \left|\frac{\partial\varphi}{\partial s}(s) ds \right|_{H^{2\beta}}^2 ds \right) dt}\\
  \vs
  \ds{\leq \frac{2 |\rho^\prime|^2_\infty}{\mu^{3 \alpha}} \int_{-\infty}^0 \int_{t - 2 \mu^{\alpha}}^t \left|\frac{\partial\varphi}{\partial s}(s) \right|_{H^{2\beta}}^2 ds dt \leq \frac{c}{\mu^{2 \alpha}} \int_{-\infty}^0 \left| \frac{\partial\varphi}{\partial s}(s) \right|_{H^{2\beta}}^2 ds.}
  \end{array}\]

 \end{proof}

The following approximation results hold.
\begin{Lemma}
Under the same assumptions of Lemma \ref{x_mu-conv-to-x-lem},
we have
  \begin{equation}
  \label{m82}
    \lim_{\mu \to 0} \left|x- \varphi_\mu(0) \right|_{H^{1+2\beta}} =0,
  \end{equation}
  and
   \begin{equation}
   \label{m83bis}
    \lim_{\mu \to 0} \sup_{t \leq 0} \left|\varphi_\mu(t) - \varphi(t) \right|_{H^{1+2\beta}} =0.
  \end{equation}
Moreover,
 \begin{equation} \label{convolution-L2-convergence-eq}
    \lim_{\mu \to 0} |\varphi_\mu - \varphi|_{L^2((-\infty,0);H^{2(1+\beta)})}=0,
  \end{equation}
  and
  \begin{equation} \label{convolution-deriv-L2-convergence-eq}
    \lim_{\mu \to 0} \left| \frac{\partial {\varphi}_\mu}{\partial t} -\frac{\partial {\varphi}}{\partial t}\right|_{L^2((-\infty,0);H^{2\beta})}=0.
  \end{equation}
\end{Lemma}

\begin{proof}
  We have
  \[
    \varphi_\mu(0) - x = \int_{-\infty}^0 \rho_\mu(-s) ( \varphi(s) - \varphi(0) ) ds,
  \]
  so that, by the continuity of $\varphi$ in $H^{1+2\beta}$, \eqref{m82} follows.

In order to prove \eqref{m83bis}, we have
  \[\left|\varphi_\mu(t) - \varphi(t) \right|_{H^{1+2\beta}} \leq \int_{-\infty}^t \rho_\mu(t-s) \left|\varphi(s) - \varphi(t) \right|_{H^{1+2\beta}} ds.
  \]
Now, as $\varphi:(-\infty,0]\to H^{1+2\beta}$ is uniformly continuous,  for any fixed $\epsilon>0$ there exists $\d_\e>0$ such that.  We use the uniform continuity of $\varphi$ to find $\delta_\epsilon>0$ such that
\[|t-s| < \delta_\epsilon\Longrightarrow     \left|\varphi(s) - \varphi(t) \right|_{H^{1+2\beta}} < \frac{\epsilon}{2}.\]
  Then if we pick $\mu$ small enough so that $\mu^\alpha < \delta_\epsilon/2$,
\[\left|\varphi_\mu(t) - \varphi(t) \right|_{H^{1+2\beta}}\leq    \int_{-\infty}^t \rho_\mu(t-s) \left|\varphi(s) - \varphi(t)\right|_{H^{1+2\beta}} ds
 \leq \int_{t - 2 \mu^\alpha}^t \frac{1}{\mu^\alpha} \frac{\epsilon}{2} = \epsilon,\]
  uniformly in $t$.  This proves \eqref{m83bis}.

   Limit \eqref{convolution-L2-convergence-eq} can be proved using the fact that
  \[\begin{array}{l}
  \ds{\left|\varphi_\mu - \varphi \right|_{L^2((-\infty,0);H^{2(1+\beta)})} =\sup_{|h |_{L^2((-\infty,0);H)}\leq 1} \int_{-\infty}^0  \left<(-A)^{1+\beta}\left( (\varphi_\mu)(t) - \varphi(t) \right), h(t) \right>_H dt}\\
  \vs
  \ds{= \sup_{|h |_{L^2((-\infty,0);H)}\leq 1} \int_{-\infty}^0 \int_{0}^{2\mu^\alpha} \rho_\mu(s) \left< (-A)^{1+\beta}\left( \varphi(t-s) - \varphi(t) \right), h(t) \right>_H ds dt}\\
  \vs
  \ds{\leq \int_{0}^{2 \mu^\alpha} \rho_\mu(s) \left|  \varphi(\cdot - s) - \varphi(\cdot)  \right|_{L^2((-\infty,0);H^{2(1+\beta)})} ds.}
  \end{array}
  \]
   Because translation is continuous in $L^2$, this converges to 0 as $\mu \downarrow 0$.
  The same argument will show that \eqref{convolution-deriv-L2-convergence-eq} holds true.

\end{proof}

Using these estimates we can prove the main result of this section.

\begin{Theorem} \label{upper-bound-thm}
  For any $x \in H^{1+2\beta}$ we have
  \begin{equation}
  \label{m84}
    \limsup_{\mu \downarrow 0} V_\mu(x) \leq V(x).
  \end{equation}
\end{Theorem}

\begin{proof}
  Let $\varphi$ be the minimizer of $V(x)$.  This means $\varphi(0)=x$, \eqref{m49} holds  and $I_{-\infty}(\varphi) = V(x).$  For each $\mu>0$, let $\varphi_\mu$ be the convolution given by \eqref{convolution-def} and let $x_\mu = \varphi_\mu(0)$.

  It is clear that
  \begin{equation} \label{tilde-V-ineq}
    V_\mu(x_\mu) \leq  I^\mu_{-\infty}(z_\mu),
  \end{equation}
  where
  \[z_\mu(t)=(\varphi_\mu(t),\frac{\partial \varphi_\mu}{\partial t}(t)),\ \ \ t\leq 0.\]
  According to Lemma \ref{x_mu-conv-to-x-lem}, we can apply
 \eqref{I-mu=I+remainder} and we have
\[\begin{array}{l}
\ds{I^\mu_{-\infty}(z_\mu) \leq \frac{c\,\mu^2}{2} \int_{-\infty}^0 |\frac{\partial^2{\varphi}_\mu}{\partial t^2}(t)|_{H^{2\beta}}^2 dt + I_{-\infty}(\varphi_\mu)}\\
\vs
\ds{
+ \mu \int_{-\infty}^0 \left<Q^{-1} \frac{\partial^2{\varphi}_\mu}{\partial t^2}(t),Q^{-1} \left( \frac{\partial{\varphi}_\mu}{\partial t}(t) - A \varphi_\mu(t) - B(\varphi_\mu(t)) \right)\right>_H dt}\\
\vs
\ds{
\leq \frac{\mu^2}{2} \int_{-\infty}^0 |\frac{\partial^2{\varphi}_\mu}{\partial t^2}(t)|_{H^{2\beta}}^2 + I_{-\infty}(\varphi_\mu)+  \mu \left( \int_{-\infty}^0 |\frac{\partial^2{\varphi}_\mu}{\partial t^2}t)|_{H^{2\beta}}^2 dt \right)^{1/2} \left( I_{-\infty}(\varphi_\mu) \right)^{1/2}.}
  \end{array}
\]
  By \eqref{m81}, this implies
  \[
I^\mu_{-\infty}(z_\mu) \leq I_{-\infty}(\varphi_\mu) + c \mu^{2-2\alpha} \left|\frac{\partial{\varphi}}{\partial t} \right|_{L^2((-\infty,0);H^{2\beta})}^2+c \mu^{1 - \alpha} \left| \frac{\partial{\varphi}}{\partial t}  \right|_{L^2((-\infty,0);H^{2\beta})} \left(I_{-\infty}(\varphi_\mu) \right)^{1/2},\]
and by  \eqref{convolution-L2-convergence-eq} and \eqref{convolution-deriv-L2-convergence-eq}
\[    \lim_{\mu \downarrow 0} I_{-\infty}(\varphi_\mu) = I_{-\infty}(\varphi) = V(x).
  \]
  Therefore, if we pick $\alpha<1$ in  \eqref{m83}, we get
  \begin{equation}
    \limsup_{\mu \downarrow 0} V_\mu(x_\mu) \leq \limsup_{\mu \downarrow 0} I^\mu_{-\infty}(z_\mu) \leq V(x).
  \end{equation}
Since, in view of \eqref{m82}  and Theorem \ref{continuity-of-V-mu-tilde-thm},
  \begin{equation*}
    \limsup_{\mu \downarrow 0} V_\mu(x_\mu) = \limsup_{\mu \downarrow 0} V_\mu(x)
  \end{equation*}
we can conclude that \eqref{m84} holds.
\end{proof}

\begin{Corollary} \label{upper-bound-cor}
  For any closed set $N \subset H$,
  \begin{equation} \label{upper-bound-cor-eq}
    \limsup_{\mu \to 0} \inf_{x \in N} V_\mu(x) \leq \inf_{x \in N} V(x)
  \end{equation}
\end{Corollary}

\begin{proof}
  If $\displaystyle \inf_{x \in N} V(x) = +\infty$ then the theorem is trivially true.  So we assume that this is not the case.  Then  by the compactness of the level sets of $V$ and the closedness of $N$, there exists $x_0 \in N$ such that $V(x_0) = \inf_{x \in N} V(x)$.
  By \eqref{m84}, we can conclude, as
  \begin{equation*}
    \limsup_{\mu \to 0} \inf_{x \in N} V_\mu (x) \leq \limsup_{\mu \downarrow 0} V_\mu(x_0) \leq V(x_0) = \inf_{x \in N} V(x).
  \end{equation*}
\end{proof}

\section{Lower bound}
\label{sec7}


Let $N \subset H$ be a closed set with $N \cap H^{1+ 2\beta} \not = \emptyset$.  In particular, by Theorem \ref{continuity-of-V-mu-tilde-thm} we have $\inf_{x \in N} V_\mu(x) < +\infty$.
Due to \eqref{m68} and Theorem \ref{compact-level-sets-half-line-thm}, there exists   $z^\mu \in\,C((-\infty,0];\H)$ such that
\[x^\mu:=\Pi_1 z^\mu(0) \in N,\ \ \ \ I^\mu_{-\infty}(z^\mu) = V_\mu(x^\mu) = \inf_{x \in N} V_\mu(x).\] Now, let $\psi^\mu \in L^2((-\infty,0);H)$ be the minimal control such that
\[z^\mu(t) = \int_{-\infty}^t \Smu(t-s)  B_\mu(z^\mu(s))  ds+ \int_{-\infty}^t \Smu(t-s) \Qmu \psi^\mu(s)  ds,\]

and
\begin{equation}
\label{l2}
  \inf_{x \in N} V_\mu(x) = V_\mu(x^\mu) = \frac{1}{2} \left| \psi^\mu \right|_{L^2((-\infty,0);H)}^2.
\end{equation}

In what follows, we shall denote $y^\mu=\Pi_2 z^\mu(0).$
For any $\delta>0$, we  define the approximate control
\[
  \psi^{\mu,\delta}(t) = (I - \delta A)^{-\frac{1}{2}} \psi^{\mu}(t),\ \ \ t\leq 0,
\]
and in view of Corollary \ref{path-existence-corollary} we can define $z^{\mu,\delta}$ to be the solution to the corresponding control problem
\[
  z^{\mu,\delta} (t) = \int_{-\infty}^t  \Smu(t-s)  B_\mu(z^{\mu,\delta}(s))  ds+ \int_{-\infty}^t  \Smu(t-s) \Qmu \psi^{\mu,\delta}(s)  ds.\]
Notice that, according to \eqref{m30},
\[\lim_{t\to-\infty}|z^{\mu,\delta}|_{\H_{1+2\beta}}=0.\]
Moreover, as $  \psi^{\mu,\delta} \in\,L^2((-\infty,0);H^1)$, thanks to \eqref{m30} we have
\[\lim_{t\to-\infty}|z^{\mu,\delta}|_{\H_{2(1+\beta)}}=0.\]
In what follows, we shall denote $(x^{\mu,\d},y^{\mu,\d})=z^{\mu,\d}(0).$

\begin{Lemma} \label{endpoint-difference-theorem}
There exists $\mu_0>0$ such that,
 \begin{equation}
 \label{m89}
    \lim_{\d\to 0}\,\sup_{\mu\leq \mu_0}\,\left| x^\mu - x^{\mu,\delta}  \right|_{H^{2\beta}}^2  =0.
  \end{equation}
\end{Lemma}
\begin{proof}%
By \eqref{m20bis}, there exists $\mu_0>0$ such that for $\mu < \mu_0$
\begin{equation*}
  |x^\mu - x^{\mu,\delta}|_{H^{2\beta}} \leq c\,|\psi^\mu - \psi^{\mu,\delta}|_{L^2((-\infty,0);H^{-1})}.
\end{equation*}

Now, since for any $h \in\,H$
 \[
    \left| (-A)^{-\frac{1}{2}}(I - \delta A)^{-\frac{1}{2}} h -(-A)^{-\frac{1}{2}} h \right|_H^2
    = \sum_{k=1}^\infty \frac{1}{\alpha_k} \left( 1 - \frac{1}{(1 + \delta \alpha_k)^{\frac{1}{2}}}  \right)^2 h_k^2,
  \]
  and
\[    \left( 1 - \frac{1}{(1 + \delta \alpha_k)^{\frac{1}{2}}}  \right)^2 \leq \alpha_k \delta,\]
we have\[    \left| (-A)^{-\frac{1}{2}}(I - \delta A)^{-\frac{1}{2}} h -(-A)^{-\frac{1}{2}} h \right|_H^2 \leq \delta |h|_H^2.
  \]
This implies
\[|x^\mu - x^{\mu,\delta}|_{H^{2\beta}}^2\leq c\,  \d\int_{-\infty}^0|\psi^\mu(s)|^2_H\,ds=c\,\d\, \inf_{x \in N} V_\mu(x).\]
In Corollary \ref{upper-bound-cor} we have proved
\[    \limsup_{\mu \downarrow 0} \inf_{x \in N} V_\mu(x) \leq \inf_{x \in N} V(x),\]
and then we obtain
\begin{equation}
\label{m100}
\sup_{\mu\leq \mu_0}\,|x^\mu- x^{\mu,\delta}|_{H^{2\beta}}\leq c\,\sqrt{\d}\,
  \end{equation}
  which implies \eqref{m89}.
\end{proof}

Now we can prove the main result of this section.
\begin{Theorem}
\label{t.82}
  For any closed $N \subset H$, we have
  \begin{equation}
  \label{m-fine101}
    \inf_{x \in N} V(x) \leq \liminf_{\mu \downarrow 0} \inf_{x \in N} V_\mu(x).
  \end{equation}

\end{Theorem}

\begin{proof}
 If the right hand side in \eqref{m-fine101} is infinite,  the theorem is trivially true.  Therefore, in what follows we can assume that
 \begin{equation}
 \label{m-fine100}
\liminf_{\mu \to 0} \inf_{x \in N} V_\mu(x) < +\infty.
 \end{equation}

  We first observe that, if we define
  \[\varphi^{\mu,\d}(t)=\Pi_1 z^{\mu,\d}(t),\ \ \ t\leq 0,\]
  in view of \eqref{I-mu=I+remainder}
\begin{equation}
 \label{V<V^mu-ineq}
 \begin{array}{l}
 \ds{V(x^{\mu,\delta}) \leq I_{-\infty}(\varphi^{\mu,\delta}) =I^\mu_{-\infty}(z^{\mu,\delta}) - \frac{\mu^2}{2} \int_{-\infty}^0 \left|Q^{-1}\frac{\partial^2{\varphi}^{\mu,\delta}}{\partial t^2}(t) \right|_H^2 dt }\\
 \vs
 \ds{- \mu \int_{-\infty}^0 \left<Q^{-1} \frac{\partial^2{\varphi}^{\mu,\delta}}{\partial t^2}(t),  Q^{-1} \frac{\partial{\varphi}^{\mu,\delta}}{\partial t}(t) - Q^{-1} A \varphi^{\mu,\delta}(t) - Q^{-1}B(\varphi^{\mu,\delta}(t)) \right>_H dt.}
    \end{array}
    \end{equation}
Since
\begin{equation}
\label{m91}
|\psi^{\mu,\d}(t)|_H=|(I-\d A)^{-1/2}\psi^\mu(t)|_H\leq |\psi^\mu(t)|_H,\ \ \ \ t\leq 0,
\end{equation}
we have
\[I^\mu_{-\infty}(z^{\mu,\delta})\leq I^\mu_{-\infty}(z^{\mu})=\inf_{x \in N}V_\mu(x),\]
so that
\[\begin{array}{l}
\ds{ V(x^{\mu,\delta}) \leq  \inf_{x \in N} V_\mu(x)} \\
\vs
\ds{- \mu \int_{-\infty}^0 \left<Q^{-1} \frac{\partial^2{\varphi}^{\mu,\delta}}{\partial t^2}(t),  Q^{-1} \frac{\partial{\varphi}^{\mu,\delta}}{\partial t}(t) - Q^{-1} A \varphi^{\mu,\delta}(t) - Q^{-1}B(\varphi^{\mu,\delta}(t)) \right>_H dt.}
\end{array}\]
Thanks to \eqref{m30} and Hypothesis \ref{H2}, by integrating by parts
\begin{equation}
\label{remainder}
\begin{array}{l}
\ds{V(x^{\mu,\d})\leq  \inf_{x \in N} V_\mu(x) }\\
\vs
\ds{-\frac{\mu}{2} |Q^{-1} y^{\mu,\delta} |_H^2
    - \mu \left<(-A) Q^{-1} x^{\mu,\delta}, Q^{-1} y^{\mu,\delta} \right>_H
    + \mu \left< Q^{-1}B(x^{\mu,\delta}), Q^{-1} y^{\mu,\delta} \right>_H}\\
    \vs
    \ds{+ c\,\mu \int_{-\infty}^0 \left| \frac{\partial{\varphi}^{\mu,\delta}}{\partial t}(t) \right|_{H^{1+2\beta}}^2 dt
    + c\,\gamma_{2\beta} \mu \int_{-\infty}^0 \left|\frac{\partial{\varphi}^{\mu,\delta}}{\partial t}(t) \right|_{H^{2\beta}}^2 dt= \inf_{x \in N}V_\mu(x)+\sum_{i=1}^5I^{\mu,\d}_i.}
    \end{array}
    \end{equation}

First, we note that
\begin{equation}
  \label{I_1-est}
  I^{\mu,\delta}_1 \leq 0.
\end{equation}
Next, by \eqref{m43} see that
\begin{equation*}
\begin{array}{l}
  \ds{I^{\mu,\delta}_2 + I^{\mu,\delta}_4 \leq c\sqrt{\mu} \left( |x^{\mu,\delta}|_{H^{2\beta +2}}^2 + \mu |y^{\mu,\delta}|_{H^{2\beta +1 }}^2  \right) + c\mu \int_{-\infty}^0 |z^{\mu,\delta}(t)|_{\H^{2 + 2\beta}}^2 dt} \\
  \ds{\leq c(\mu + \sqrt{\mu}) \int_{-\infty}^0 |\psi^{\mu,\delta}(t)|_{H^1}^2 dt}.
\end{array}
\end{equation*}
Since for any $h \in\,H$
we have $(I - \delta A)^{-\frac{1}{2}} h \in D(-A)^{\frac{1}{2}}$ and
\[\left| (-A)^{\frac{1}{2}} (I - \delta A)^{-\frac{1}{2}} h \right|_H \leq \delta^{-1/2} \left|h \right|_H,\]
we have
\[\left|\psi^{\mu,\delta}(t) \right|_{H^1}\leq \d^{-1/2}\,\left|\psi^{\mu}(t) \right|_{H},\ \ \ t\leq 0.\]
Therefore, by \eqref{l2},
\begin{equation}
  \label{I_2-I_4-est}
  I^{\mu,\delta}_2 + I^{\mu,\delta}_4 \leq c\, \d^{-1/2}(\mu + \sqrt{\mu}) \int_0^t |\psi^\mu(t)|_H^2 =2\, c \delta^{-1/2}(\mu + \sqrt{\mu}) \inf_{x \in N} V_\mu(x).
\end{equation}
By the same arguments, \eqref{m43}, and \eqref{m91} give
\begin{equation}
  \label{I_3-I_5-est}
  I^{\mu,\delta}_3 + I^{\mu,\delta}_5 \leq c( \mu + \sqrt{\mu}) \inf_{x \in N} V_\mu(x).
\end{equation}
Combining together \eqref{I_1-est}, \eqref{I_2-I_4-est}, and \eqref{I_3-I_5-est} with \eqref{remainder},  we obtain,
\begin{equation}
  V(x^{\mu,\delta}) \leq \inf_{x \in N} V_\mu(x)+ c(\mu + \sqrt{\mu})(1 + \delta^{-1/2})\, \inf_{x \in N} V_\mu(x).
\end{equation}
From this, due to \eqref{m-fine100} we see that
\[\liminf_{\mu\to0}V(x^{\mu,\sqrt{\mu}})\leq \liminf_{\mu\to0} \inf_{x \in N}V_\mu(x).\]
Since we are assuming \eqref{m-fine100}, and, by \cite[Proposition 5.1]{cerrok}, the level sets of $V$ are compact, there is a sequence $\mu_n \to 0$ and $x^0 \in\,H$ such that
\[\lim_{n\to \infty}|x^{\mu_n,\sqrt{\mu_n}}-x^0|_H=0,\ \ \ V(x^0) \leq \liminf_{\mu \to 0} V(x^{\mu,\sqrt{\mu}}).\]
By \eqref{m89}, we have that $x^{\mu_n}$ converges to $x^0$ in $H$, so that $x_0 \in N$.  This means that  we can conclude, as
\[\inf_{x \in N} V(x) \leq V(x^0) \leq \liminf_{\mu \to 0} V(x^{\mu,\sqrt{\mu}}) \leq \liminf_{\mu \to 0} \inf_{x \in N} V_\mu(x).   \]

\end{proof}

\section{Application to the exit problem}
In this section we  study the problem of the exit of the solution $u^\mu_\e$ of equation \eqref{semilinear-wave-eq-intro} from a domain $G\subset H$, for any $\mu>0$ fixed. Then we  apply the limiting results proved in Theorems \ref{upper-bound-thm} and \ref{t.82} to show that, when $\mu$ is small, the relevant quantities in the exit problem from $G$ for the solution $u^\mu_\e$ of equation \eqref{semilinear-wave-eq-intro}  can be approximated by the corresponding ones arising  for equation \eqref{semilinear-heat-eq-intro}.

First, let us give some assumptions on the set $G$. 

\begin{Hypothesis}
\label{H3}
The domain $G \subset H$ is an open, bounded, connected set, such that  $0 \in G$.
Moreover, for any $x \in \partial G \cap H^{1 + 2\beta}$ there exists a sequence $\{x_n\}_{n \in \nat} \subset \bar{G}^c \cap H^{1+2\beta}$ such that
\begin{equation} \label{boundary-reg-assum}
  \lim_{n \to +\infty} \left| x_n - x \right|_{\H_{1 + 2\beta}} =0.
\end{equation}
\end{Hypothesis}
Assume now that $G$ is an open, bounded and connected set such that, for any $x \in \partial G \cap H^{1 + 2\beta}$, there exists a $y \in \bar{G}^c \cap H^{1 + 2\beta}$ such that
\begin{equation}
\label{r.1}
\{ty + (1-t)x: 0<t\leq1\} \subset \bar{G}^c.\end{equation}
Then it is immediate to check that \eqref{boundary-reg-assum} is satisfied.
Condition \eqref{r.1}  is true, for example, if $G$ is convex, because of the Hahn-Banach separation theorem and the density of $H^{1 + 2\beta}$ in $H$.

\begin{Lemma}
Under Hypothesis \ref{H3}
\begin{equation}
\label{m-partial}
V_\mu(\partial G):=  \inf_{x \in \partial G} V_\mu(x)=V_\mu(x_{G,\mu})<\infty,
\end{equation}
for some $x_{G,\mu} \in\,\partial G\cap \H_{1+2\beta}$.
\end{Lemma}

\begin{proof}
Since  $\bar{G}^c$ is an open set,  there exists $\tilde{x} \in \bar{G}^c \cap H^{1 + 2\beta}$. Because $0 \in G$, and the path $t \mapsto t\tilde{x}$ is continuous, there must exist $0<t_0<1$ such that $ t_0\tilde{x} \in \partial G$. Clearly, $t_0 \tilde{x} \in H^{1 + 2\beta}$, so that, according to Theorem \ref{continuity-of-V-mu-tilde-thm}, the first equality in \eqref{m-partial} is true as $\partial G\cap H^{1+2\beta}\neq \emptyset$.

 Moreover, thanks to Theorem \ref{fine1}, the first equality in \eqref{m-partial} implies that there exists $x_{G,\mu} \in\,\partial G\cap \H_{1+2\beta}$ such that
\begin{equation}
\label{m-minimum}
V_\mu(x_{G,\mu})=V_\mu(\partial G).
\end{equation}

\end{proof}

Now, if we denote by $z^\mu_{\e,z_0}=(u^\mu_{\e,z_0},v^\mu_{\e,z_0})$ the mild solution of \eqref{abstract}, with initial position and velocity $z_0=(u_0,v_0) \in\,\H$, we define the exit time
\begin{equation} \label{wave-exit-time-def}
  \tau^{\mu,\epsilon}_{z_0} = \inf \left\{t>0: u^\mu_{\epsilon,z_0} (t) \not \in G \right\}.
\end{equation}
Here is the main result of this section
\begin{Theorem}
\label{m-t81}
  There exists $\mu_0>0$ such that for $\mu<\mu_0$ the following conditions are verified.
  For any $z_0=(u_0,v_0) \in\,\H$ such that $u_0 \in\,G$ and $u^\mu_{0,z_0}(t) \in\,G$, for $t\geq 0$,
  \begin{enumerate}
    \item The exit time has the following asymptotic growth
    \begin{equation}
    \label{m-con1}
      \lim_{\epsilon \to 0} \epsilon \log  \E \left( \tau^{\mu,\epsilon}_{z_0} \right)= \inf_{x \in \partial G} V_\mu(x),
    \end{equation}
  and for any any $\eta>0$,
    \begin{equation}
    \label{m-con2}
      \lim_{\epsilon \to 0} \Pro \left( \exp\le(\frac{1}{\epsilon}(V_\mu(\partial G) -\eta)\r) \leq \tau^{\mu,\epsilon}_{z_0} \leq \exp\le(\frac{1}{\epsilon}(V_\mu(\partial G) + \eta)\r) \right) =1.
    \end{equation}
    
    \item  For any closed $N \subset \partial G$ such that $\ds{\inf_{x\in N}V_\mu(x) > \inf_{x \in \partial G} V_\mu (x)}$, it holds
    \begin{equation} \label{m-con3}
      \lim_{\epsilon \to 0} \Pro\left( u^\mu_{\epsilon,z_0}(\tau^{\mu,\epsilon}_{z_0}) \in N \right) = 0.
    \end{equation}
  \end{enumerate}
\end{Theorem}

\begin{Remark}
{\em   The requirement that $u^\mu_{0,z_0}(t) \in G$ for all $t \geq 0$ is necessary because in Lemma \ref{large-init-veloc-causes-exit-lem} we showed that there exist $z_0\in G \times H^{-1}$ such that $u^\mu_{0,z_0}$ leaves $G$ in finite time.  Of course, for these initial conditions, the stochastic processes $u^\mu_{\epsilon,z_0}$ will also exit in finite time for small $\epsilon$.}
\end{Remark}

In \cite{tran} it has been proven that an analogous result to Theorem \ref{m-t81} holds for equation \eqref{abstract-heat}.
If we denote by $u_{\epsilon,u_0}$  the mild solutions of equation \eqref{abstract-heat}, with initial condition $u_0 \in\,H$, we define  the exit time
\[
  \tau^\epsilon_{u_0} = \inf \left\{t>0: u_{\epsilon,u_0}(t) \not \in G \right\}.
\]
In \cite{tran} it has been proven that for any $u_0 \in\,G$ such that $u_{0,u_0}(t) \in\,G$, for any $t\geq 0$, it holds
\[
\lim_{\epsilon \to 0} \epsilon \log  \E \left( \tau^{\epsilon}_{u_0} \right)= \inf_{x \in \partial G} V(x).
\]
Similarly, as we would expect, it also holds that
\[\lim_{\e \to 0} \e \log \tau^\e_{u_0} = \inf_{x \in \partial G} V(x),\ \ \ \text{in probability},\]
and if $N \subset \partial G$ is closed and $\inf_{x \in N} V(x)> \inf_{x \in \partial G} V(x)$,
\[\lim_{\e \to 0} \Pro \left(u^\e_{u_0}(\tau^\e_{u_0}) \in N \right) =0.\]
The proof of these facts is analogous to the proof of Theorem \ref{m-t81}

In view of what we have proven in Sections \ref{sec6} and \ref{sec7} and of Theorem \ref{m-t81}, this implies   that
the following Smoluchowski-Kramers approximations holds for the exit time.

\begin{Theorem}
\begin{enumerate}
  \item
  For any initial conditions $z_0 = (u_0,v_0)$,
  \begin{equation}
    \lim_{\mu \to 0} \lim_{\epsilon \to 0} \e\log\,\E \left(\tau^{\mu,\epsilon}_{z_0} \right) = \lim_{\epsilon \to 0} \e\log\,\E \left( \tau^{\epsilon}_{u_0} \right)=\inf_{x \in\,\partial G}V(x).
  \end{equation}

  \item  For any $\eta>0$, there exists $\mu_0>0$ such that for $\mu< \mu_0$
  \begin{equation}
  \label{m-fine200}
    \lim_{\epsilon \to 0} \Pro \left( e^{\frac{1}{\epsilon}(\bar{V} -\eta)} \leq \tau^{\mu,\epsilon}_{z_0} \leq e^{\frac{1}{\epsilon}(\bar{V} + \eta)}  \right) =1
  \end{equation}

  \item  For any $N \subset \partial G$ such that $\inf_{x \in N} V(x) < \inf_{x \in \partial G} V(x)$, there exits $\mu_0>0$ such that for all $\mu < \mu_0$,
  \begin{equation*}
    \lim_{\epsilon \to 0} \Pro_{z_0} \left( u^\mu_\epsilon(\tau^{\mu,\epsilon}) \in N \right) = 0.
  \end{equation*}
\end{enumerate}
\end{Theorem}

We recall that in \cite{sal} we have proved that, in the case of gradient systems, for any $\mu>0$
\[V_\mu(x)=V(x),\ \ \ \ x \in\,H.\]
This means that in this case for any $z_0=(u_0,v_0) \in\,\H$ and $\mu>0$
\[\lim_{\epsilon \to 0} \e\,\log \E \left(\tau^{\mu,\epsilon}_{z_0} \right) = \lim_{\epsilon \to 0} \e\,\log \E \left( \tau^{\epsilon}_{u_0} \right)=\inf_{x \in\,\partial G}V(x).\]
and \eqref{m-fine200} holds for any $\mu>0$.

\subsection{Proof to Theorem \ref{m-t81}}

In order to prove Theorem \ref{m-t81}, we will need some preliminary lemmas, whose proofs are postponed to the next subsection.

\begin{Lemma}
\label{lemma1}
  For $\mu < (\alpha_1 - \gamma_0)\gamma_0^{-2}$, there exists a constant $c(\mu)>0$ such that  $z_1,z_2 \in \H$
  \begin{equation} \label{continuity-of-control-system}
\sup_{\psi \in\,L^2((0,+\infty);H)}    \sup_{t \ge 0} \left| z^\mu_{\psi,z_1}(t) - z^{\mu}_{\psi,z_2}(t) \right|_\H \leq c(\mu) |z_1 - z_2|_\H.
  \end{equation}
\end{Lemma}

\begin{Lemma} \label{far-from-N-lem}
  For any closed set $N \subset H$, and any $A< V_\mu(N)$, there exists $\rho_0>0$  such that if $z \in C((0,T);H)$, with $|z(0)|_\H<\rho_0$ and $I^\mu_{0,T}(z) < A$, then it holds
  \begin{equation*}
    \inf_{t \leq T} \textnormal{dist}_H(\Pi_1 z(t), N) > |z(0)|_\H.
  \end{equation*}
\end{Lemma}

\begin{Lemma} \label{stay-out-of-small-ball-lem}
For any $\mu,\e>0$ and $z_0 \in\,\H$, let
  \[\tau^{\mu,\epsilon}_{z_0,\rho} := \inf \left\{t>0: \Pi_1 z^{\mu}_{\epsilon,z_0}(t) \not \in G \text{ or } \left|z^{\mu}_{\epsilon,z_0} \right|_\H < \rho \right\},\]
where $\rho>0$ is small enough so that $B_{\H}(\rho)\subset G \times H^{-1}$.  Then
  \begin{equation} \label{stay-out-of-small-ball-eq}
    \lim_{t \to +\infty} \limsup_{\epsilon \to 0} \epsilon \log \left( \sup_{z_0 \in G \times H^{-1}} \P \left( \tau^{\mu,\epsilon}_{z_0,\rho} \ge t \right) \right) = -\infty.
  \end{equation}
\end{Lemma}

\begin{Lemma}
\label{lemma4}
  Let $\tau^{\mu,\epsilon}_{z_0,\rho}$ be the exit time from Lemma \ref{stay-out-of-small-ball-lem} and let $N \subset \partial G$ be a closed set.  Then
  \begin{equation}
  \label{hit-N-prob-bound}
    \lim_{\rho \to 0}\, \limsup_{\epsilon \to 0} \epsilon \log \left( \sup_{z_0 \in\,B_\H((1 + M_\mu) \rho)} \P\left( \Pi_1  z^{\mu}_{\epsilon,z_0}
    \left(\tau^{\mu,\epsilon}_{z_0,\rho}\right) \in N \right)  \right) \leq - V_\mu(N),
  \end{equation}
  where $V_\mu(N)=\inf_{x \in\,N}V_\mu(x)$.
\end{Lemma}

\begin{Lemma}
\label{lemma5}
  For fixed $\rho>0$,
  \begin{equation*} \label{time-lower-bound-eq}
    \lim_{t \to 0} \limsup_{\epsilon \to 0} \epsilon \log \left( \sup_{z_0 \in\,B_{\H}(\rho)} \P \left( \sup_{s \leq t} \left| z^{\mu}_{\epsilon,z_0}(s) \right|_\H \ge (1 + M_\mu) \rho \right) \right) = -\infty.
  \end{equation*}

\end{Lemma}

\medskip

\begin{proof}[Proof of Theorem \ref{m-t81}]   As $G \subset H$ is a bounded set, there exists $R>0$ such that $G\subset B_H(R-1)$
If $c(\mu,1)$ is the constant from Lemma \ref{large-init-veloc-causes-exit-lem}, for any $z_0=(u_0,v_0) \in\,\H$ such that
\[u_0 \in\,G,\ \ \ \ |v_0|_{H^{-1}}>R\,c(\mu,1)^{-1}=:\kappa,\]  we have that  $\Pi_1 z^\mu_{z_0}$ leaves $B_{R}$, (and therefore  $G$) before time $t=1$.
Since for any $T>0$
\begin{equation}
\label{m-fine225}
\lim_{\e\to 0}\sup_{z_0 \in\,\H}\E\,|z^\mu_{\e,z_0}-z^\mu_{z_0}|_{C([0,T];\H)}=0,
\end{equation}
this yields
  \begin{equation}
   \label{exit-time-finite-in-prob-eq}
 \begin{array}{l}
 \ds{   \lim_{\epsilon \to 0} \inf_{\substack{ u_0 \in G \\ |v_0|_{H^{-1}} > \kappa}} \P \left( \tau^{\mu,\epsilon}_{z_0} < 1 \right) \geq  \lim_{\epsilon \to 0} \inf_{\substack{ u_0 \in G \\ |v_0|_{H^{-1}} > \kappa}} \P \left(  \left| z^\mu_{\e,z_0}- z^\mu_{z_0}\right|_{C([0,T];\H)} \leq 1 \right) = 1.}
  \end{array}
  \end{equation}

 Now, fix $\eta>0$.  According to \eqref{m-minimum}, there exists $x_{G,\mu} \in\,\partial G\cap H^{1+2\beta}$ such that $V_\mu(x_{G,\mu})=V_\mu(\partial G)$. Now, if $\{x_n\} \subset \bar{G}^c\cap H^{1+2\beta}$ is a sequence from \eqref{boundary-reg-assum} such that $x_n\to x_{G,\mu}$ in $H^{1+2\beta}$, as $n\to\infty$,  due to Theorem \ref{continuity-of-V-mu-tilde-thm} we have that $V_\mu(x_n)\to V_\mu(x_{G,\mu})$. This means that there exists $\bar{n}$ such that
\[V_\mu(x_{\bar{n}})<V_\mu(x_{G,\mu})+\frac \eta 4=V_\mu(\partial G)+\frac \eta 4.\]
In particular, there exists
$T_1>0$ and $z^{\mu}_{\psi,0} \in C([0,T_1];\H)$ such that $z^{\mu}_{\psi,0} (0) = 0$ and $\Pi_1 z^{\mu}_{\psi,0} (T_1)=x_{\bar{n}} \in \bar{G}^c$  with
 \[I^\mu_{0,T_1}(z^{\mu}_{\psi,0} ) < V_\mu(x_{\bar{n}})+\frac \eta 4<V_\mu(\partial G) + \frac{\eta}{2}.\]

According to \eqref{continuity-of-control-system}, the mapping $z_0 \in\,\H \mapsto \varphi^\mu_{\psi,z_0} \in C([0,T_1];\H)$ is continuous, and therefore, we can find  $\rho>0$ such that
  \begin{equation*}
    |z_0|_\H < \rho \Longrightarrow \text{dist}\left( z^\mu_{\psi,z_0}(T_1), (G \times H^{-1}) \right) > \frac 12 \text{dist}\left( z^\mu_{\psi,0}(T_1), (G \times H^{-1}) \right)=:\a>0.
  \end{equation*}
In view of \eqref{LDP-lower-bound}, we can see that there exists $\epsilon_1>0$ such that for all $\epsilon< \epsilon_1$, and all $|z_0|_\H<\rho$
  \begin{equation} \label{exit-prob-small-init-cond}
    \P \left( \tau^{\mu,\epsilon}_{z_0} < T_1 \right) \ge \P \left( \left| z^{\mu}_{\epsilon,z_0} - z^{\mu}_{\psi,z_0} \right|_{C([0,T_1];\H)} < \a \right) \ge e^{-\frac{1}{\epsilon}( V_\mu(G) + \eta)}.
  \end{equation}
 Now, by Lemma \ref{unperturbed-system-conv-to-zero-lem} we can find $T_2>$ such that
  \begin{equation*}
    \sup_{\substack{u_0 \in G \\ |v_0|_{H^{-1}} \leq \kappa}} \left| z^{\mu}_{z_0}(T_2) \right|_\H < \frac{\rho}{2}.
  \end{equation*}
Therefore, thanks to \eqref{m-fine225}, there exists $0<\e_2\leq \e_1$ such that
 $u_0 \in G$, and $|v_0|_{H^{-1}} \leq \kappa$,
  \begin{equation*}
     \P \left( \left|z^{\mu}_{\epsilon,z_0}(T_2) \right|_\H < \rho \right) > \frac{1}{2},\ \ \ \ \e\leq \e_2.
  \end{equation*}
Thanks to \eqref{exit-prob-small-init-cond}, by the
Markov property, this implies that for $u_0 \in G$ and $|v_0|_{H^{-1}} \leq \kappa$,
  \begin{equation*}
     \P \left( \tau^{\mu,\epsilon}_{z_0} < T_1 + T_2 \right) \ge \frac{1}{2} e^{-\frac{1}{\epsilon} (V_\mu(G) + \eta)},\ \ \ \ \e<\e_2.
  \end{equation*}
Hence, if we combine this with \eqref{exit-time-finite-in-prob-eq}, we see that there exists $0<\epsilon_0\leq \e_2$ such that for all $\epsilon< \epsilon_0$,
  \begin{equation}
    \inf_{z_0 \in G \times H^{-1}} \P \left( \tau^{\mu,\epsilon}_{z_0} < 1+ T_1 + T_2 \right) \ge \frac{1}{2} e^{-\frac{1}{\epsilon} (V_\mu(G) + \eta)}.
  \end{equation}

By using again the Markov property, for any $k \in \mathbb{N}$ and  $z_0 \in G \times H^{-1}$ this gives
  \[
\begin{array}{l}
\ds{\P \left( \tau^{\mu,\epsilon}_{z_0} \ge k ( 1+ T_1 + T_2) \right)
     \leq \left(\sup_{z_0 \in G \times H^{-1}} \P( \tau^{\mu,\epsilon}_{z_0} \ge (1+ T_1 + T_2)) \right)^k
     \leq  \left( 1 - \frac{1}{2} e^{-\frac{1}{\epsilon} (V_\mu(G) + \eta)} \right)^k,}
  \end{array}\]
so that
  \[\begin{array}{l}
\ds{\E \left( \tau^{\mu,\epsilon}_z \right) \leq (1+T_1+T_2) \sum_{k=0}^\infty \P \left( \tau^{\mu,\epsilon}_z \ge k(1 + T_1 +T_2) \right)\leq 2(1+T_1+T_2) e^{\frac{1}{\epsilon}(V_\mu(G) + \eta)}.}
\end{array}
  \]

  Thus, the upper bound of \eqref{m-con1} follows as $\eta$ was chosen arbitrarily small and
the upper bound of \eqref{m-con2}, follows from this by using the Chebyshev inequality.
\end{proof}

\medskip

The proofs of the lower bound for the exit time and of the exit place  follow from Lemmas \ref{lemma1} to \ref{lemma5}, by using the same arguments used in the finite dimensional case (see  \cite{dz} and \cite{fw}). For this reason, we omit them.

\subsection{Proofs of Lemmas from \ref{lemma1} to \ref{lemma5}}

\begin{proof}[Proof of Lemma \ref{lemma1}]  If we let $\varphi(t) = \Pi_1 \left( z^\mu_{\psi,z_1}(t) - z^{\mu}_{\psi,z_2}(t) \right)$, then it is a weak solution to
  \begin{equation}
    \mu \frac{\partial^2 \varphi}{\partial t^2}(t) + \frac{\partial \varphi}{\partial t}(t) = A \varphi(t) + B(\Pi_1 z^{\mu}_{z_1,\psi}(t)) - B(\Pi_1 z^{\mu}_{z_2,\psi}(t)).
  \end{equation}
Therefore, we can conclude as in  Lemma \ref{X-sup-bounded-lemma}.
\end{proof}

\begin{proof}[Proof of Lemma \ref{far-from-N-lem}]  Fix $A< V_\mu(N)$. Suppose by contradiction that there exist $\{z_n\}\subset \H$, $\{T_n\}\subset (0,+\infty)$ and $\{\psi_n\}\subset L^2((0,T_n);H)$ such that
\[\lim_{n\to \infty}|z_n|_\H=0,\ \ \ \ \ \frac{1}{2} |\psi_n|_{L^2((0,T_n);H)}^2 < A,\]
and
\[\text{dist}_H(\Pi_1 z^\mu_{\psi_n,z_n}(T_n), N) \leq |z_n|_\H.\]

Now, if we set $x_n:=\Pi_1 z^\mu_{\psi_n,0}(T_n)$, for any $n \in\,\nat$ we have, by \eqref{continuity-of-control-system},
\[|x_n-\Pi_1 z^\mu_{\psi_n,z_n}(T_n)|_H\leq c(\mu)\, |z_n|_\H,\]
so that
  \begin{equation} \label{dist-from-N-eq}
    \text{dist}_H ( x_n, N) \leq c(\mu) |z_n|_\H+ |z_n|_\H,\ \ \ \ n \in\,\nat.
  \end{equation}
Recalling how $V_\mu$ is defined, this implies
  \begin{equation*}
    V_\mu(x_n) \leq \frac{1}{2} |\psi_n|^2_{L^2((0,T_n);H)} < A.
  \end{equation*}

 Now, as proven in Theorem \ref{fine1}, $V_\mu$ has compact level sets.  Therefore, there is a sequence  $\{x_{n_k}\} _k\subset H$ such that $x_{n_k}\to x$, so that $V_\mu(x) < A$.  But, by \eqref{dist-from-N-eq}, $x \in N$, and then $V_\mu(N) \leq V_\mu(x)<V_\mu(N)$, a contradiction.
\end{proof}

\begin{proof}[Proof of Lemma \ref{stay-out-of-small-ball-lem}]   Fix $R>\sup_{x \in G} |x|_H+\rho$ and, by Lemma \ref{large-init-veloc-causes-exit-lem}, let us take $\kappa>0$ such that if $v_0 \in\,B_{H^{-1}}(\kappa)$ then  $z^{\mu}_{z_0}$  leaves $B_R \times H^{-1}$ before time $t=1$.

By Lemma \ref{unperturbed-system-conv-to-zero-lem}, we can find $T_1>0$ such that
  \begin{equation*}
    \sup_{\substack{u_0 \in G \\ |v_0|_{H^{-1}} \leq \kappa}} \left| z^{\mu}_{z_0}(T_1) \right|_{\H} < \frac{\rho}{2},
  \end{equation*}
and then for any $z_0 \in G \times H^{-1}$,  $z^{\mu}_{z_0}(t)$ leaves $(G \times H^{-1} )\setminus B_{\H}(\rho/2)$ in less than time $T=T_1+1$.  This means that
  \begin{equation}
    \inf \left\{ I^\mu_{0,T}(z): z(t) \in (B_H(R) \times H^{-1}) \setminus B_{\H}(\rho/2) \text{ for } t \in [0,T]  \right\} =a >0
  \end{equation}
  because the set above contains no unperturbed trajectories.
By \eqref{LDP-upper-bound}
 \[\begin{array}{l}
\ds{\limsup_{\epsilon \to 0}  \epsilon \log \left( \sup_{z_0 \in G \times H^{-1}} \P \left( \tau_0 \ge T \right) \right)} \\
\vs
\ds{\leq \limsup_{\epsilon \to 0} \epsilon \log \left( \sup_{z_0 \in G \times H^{-1}} \P \left( \text{dist}_{C([0,T];\H) }(z^{\mu}_{\epsilon,z_0}, K^\mu_{0,T}(a)) > \frac{\rho}{2} \right) \right)\leq -a.}
\end{array}\]
  By the Markov property, for any $k \in \mathbb{N}$,
  \begin{equation*}
    \sup_{z_0 \in G \times H^{-1}} \P(\tau_0 \ge kT) \leq \left(\sup_{z_0 \in G \times H^{-1}} \P(\tau_0 \ge T) \right)^k
  \end{equation*}
  and therefore,
  \begin{equation*}
    \lim_{\epsilon \to 0}  \epsilon \log \left( \sup_{z_0 \in G \times H^{-1}} \P \left( \tau_1 \ge Tk \right) \right) \leq -ka.
  \end{equation*}

\end{proof}

\begin{proof}[Proof of Lemma \ref{lemma4}]   Let $\Gamma_\rho := B_{\H}((1 + M_\mu)\rho).$
For any $T>0$ we have
\begin{equation}
\label{m-con4}  \begin{array}{l}
\ds{\sup_{z_0 \in \Gamma_\rho} \P\left( \Pi_1 z^{\mu}_{\epsilon,z_0} (\tau^{\mu,\epsilon}_{z_0,\rho}) \in N \right) \leq \sup_{z_0 \in \Gamma_\rho} \P (\tau^{\mu,\epsilon}_{z_0,\rho} >T) + \sup_{z_0 \in \Gamma_\rho} \P( \Pi_1 x^{\mu}_{\epsilon,z_0}}(t) \in N, \text{ for some } t \leq T).
  \end{array}\end{equation}

Next, thanks to Lemma \ref{far-from-N-lem}, for any $A< V_\mu(N)$ fixed we can find $\rho_0>0$ such that for $\rho<\rho_0$ and any $T>0$, the set
\[    \left\{ z: z(0) \in \Gamma_\rho, \, \text{dist}_{C([0,T];\H)} \le(z, K^\mu_{0,T}(A)\r) \leq (1+M_\mu)\rho \right\}
  \]
  contains no trajectories that reach $N$ by time $T$.
Then by \eqref{LDP-upper-bound}, for any $\eta>0$, for small enough $\epsilon>0$,
\[  \begin{array}{l}
\ds{\sup_{z_0 \in \Gamma_\rho} \P(\Pi_1 z^{\mu}_{\epsilon,z_0}(t) \in N \text{ for some } t \leq T)}\\
\vs
\ds{
\leq \sup_{z_0 \in \Gamma_\rho} \P \le( \text{dist}_{C([0,T];\H)} ( z^{\mu}_{\epsilon,z_0}, K^\mu_{0,T}(A)) > (1 + M_\mu) \rho\r) \leq e^{-\frac{1}{\epsilon}(A-\eta)}.}
  \end{array} \]
Now, according to \eqref{stay-out-of-small-ball-eq}, we pick $T>0$ so that,  for small enough $\epsilon>0$,
\[
    \sup_{z_0 \in \Gamma_\rho} \P (\tau^{\mu,\epsilon}_{z_0,\rho}>T) \leq e^{-\frac{1}{\epsilon}(A)}.
  \]
Due to \eqref{m-con4}, this implies our result,
as $A <V_\mu(N)$ and $\eta>0$ were arbitrary.
\end{proof}

\begin{proof}[Proof of Lemma \ref{lemma5}]  If $z(t) = z^\mu_{\psi,z_0}(t)$, then
  \[
    z(t) = \Smu(t) z_0 + \int_0^t \Smu(t-s) B_\mu(z(s)) ds + \int_0^t \Smu(t-s) \Qmu \psi(s) ds,
  \]
so that, if $z_0 \in\,B_{\H}(\rho)$,
\[\begin{array}{l}
\ds{\sup_{s \leq t} |z(s)|_\H \leq M_\mu \rho + \frac{\gamma_0 t M_\mu}{\mu \sqrt{\alpha_1}} \sup_{s \leq t} |z(s)|_\H+ \frac{M_\mu\|Q\|_{L(H)}}{\mu\sqrt{\alpha_1}}\sqrt{t} |\psi|_{L^2((0,t);H)}.}
  \end{array}\]
Therefore, if $\displaystyle \sup_{s \leq t} |z(s)| \ge \left(M_\mu+1/2 \right) \rho$, then we get
  \begin{equation*}
    E_\mu(t):=\rho \left(\frac{1}{2} - \frac{\gamma_0 tM_\mu}{\sqrt{\alpha_1}\mu} \right) \frac{\mu\sqrt{\alpha_1}}{M_\mu\sqrt{t}}\leq | \psi |_{L^2((0,t);H)}.
  \end{equation*}
This means that
\[\begin{array}{l}
\ds{\limsup_{\epsilon \to 0} \epsilon \log \left(\sup_{z_0 \in\,B_{\H}(\rho)} \P \left( \sup_{s \leq t} \left| z^{\mu}_{\epsilon,z_0}(s) \right|_\H \geq (1 + M_\mu) \rho \right)  \right)}\\
 \vs
 \ds{\leq \limsup_{\epsilon \to 0} \epsilon \log \left( \sup_{z_0 \in\,B_{\H}(\rho)} \P \left( \text{dist}_{C([0,t];\H)} \left(z^{\mu}_{\epsilon,z_0}, K^\mu_{0,t} \left(\frac{1}{2} (E_\mu(t))^2 \right) \right) > \frac{\rho}{2} \right) \right) \leq -\frac{(E_\mu(t))^2}{2},}
  \end{array}\]
  and our result follows as
  \[\lim_{t \to 0} E_\mu(t) = +\infty.\]
\end{proof}

\end{document}